\documentclass{amsart}
\usepackage{amsmath,amstext,amssymb,amsfonts,amscd,graphicx}
\usepackage{mathrsfs,accents,mathtools,bbm}
\newtheorem{thm}{Theorem}[section]
\newtheorem{cor}[thm]{Corollary}
\newtheorem{lem}[thm]{Lemma}

\theoremstyle{definition}

\theoremstyle{remark}
\newtheorem{rem}[thm]{Remark}
\numberwithin{equation}{section}
\newcommand{\sgn}{\mathop{\mathrm{sgn}}}

\newcommand{\lr}[1]{\langle #1 \rangle}

\newcommand\eqdef{\stackrel{\mathclap{\normalfont\mbox{def}}}{=}}
\newcommand{\ti}[1]{\tilde{#1}}

\title[Regularizing estimates for the gain term]{Sharp regularizing estimates 
for the gain term of the Boltzmann collision operator}
\author{Jin-Cheng Jiang}
\address{Department of Mathematics, National Tsing Hua University, Hsinchu, Taiwan 30013, R.O.C}
\email{jcjiang@math.nthu.edu.tw}

\begin{document}

\subjclass[2010]{Primary 35Q20; Secondary 35S30}

\keywords{Boltzmann collision operator, Gain term, Regularizing,  hard sphere, hard potential, Maxwell molecule, Fourier integral operator}

\begin{abstract}
We prove the sharp regularizing estimates for the gain term of the Boltzmann collision operator
including hard sphere, hard potential and Maxwell molecule models.
Our new estimates characterize both regularization and convolution properties of the gain term  
and have the following features. The regularizing exponent is sharp both in the $L^2$ based
inhomogeneous and homogeneous Sobolev spaces which is exact the exponent of the kinetic part of 
collision kernel. The functions in these estimates belong to a wider scope of (weighted) Lebesgue spaces   
than the previous regularizing estimates.  For the estimates in homogeneous Sobolev spaces, never seen before, 
only need functions lying in Lebesgue spaces instead of weighted Lebesgue spaces, i.e., no loss of weight 
occurs in this case.

\end{abstract}

\maketitle

\section{Introduction}

The Boltzmann collision operator reads
\begin{equation}
Q(f,f)(v)=\int_{\mathbb{R}^3}\int_{\omega\in S^2_+}(f'f'_{*}-ff_{*})B(v-v_{*},\omega)d\Omega(\omega)dv_{*}
\end{equation}
where $d\Omega(\omega)$ is the solid element in the direction of $\omega$ and
the abbreviations $f'=f(x,v',t)\;,\;f'_{*}=f(x,v'_{*},t)\;,\;f_{*}=f(x,v_{*},t)$, where
\begin{equation}\label{E:pre-collision}
v'=v-[\omega\cdot(v-v_{*})]\omega\;,\;v'_{*}=v_{*}+[\omega\cdot(v-v_{*})]\omega\;,\;\omega\in S^2_+
\end{equation}
stand for the pre-collision velocities of particles which after collision have velocities $v$
and $v_{*}$. In the study of the Boltzmann equation, one of the most important tasks is to understand the
properties of the collision operator.   

When Grad's cutoff assumption 
$$\int_{S^2_+} B(v-v_{*},\omega)d\Omega(\omega)<\infty$$ 
is satisfied, we can split the collision operator into gain and loss terms, namely
$$
Q^+(f,f)=\int_{\mathbb{R}^3}dv_{*}\int_{S^2_+}B(v-v_{*},\omega)f'f'_{*}d\Omega(\omega)
$$
$$
Q^-(f,f)=\int_{\mathbb{R}^3}dv_{*}\int_{S^2_+}B(v-v_{*},\omega)ff_{*}d\Omega(\omega).
$$
The loss term is in fact $$f(Lf)$$ where $L$ is a convolution operator in velocity variable,
while the gain term is more complicated. To study the gain term, it is convenient to consider the quadratic operator, i.e.
\begin{equation}\label{D:bi-Q-plus}
Q^+(f,g)=\int_{\mathbb{R}^3}\int_{S^2_+}B(v-v_{*},\omega)f(v')g(v'_{*})d\Omega(\omega)dv_{*}.
\end{equation}
In this paper, we consider the collision kernels of the form 
\begin{equation}\label{D:kernel-B}
\begin{split}
& B(v-v_*,\omega)=|v-v_*|^{\gamma}b(\cos\theta)=|v-v_*|^{\gamma}\cos\theta,\\
& 0\leq\gamma\leq 1\;,\;0\leq\theta\leq \pi/2
\end{split}
\end{equation}
where $\cos\theta=(v-v_*)\cdot\omega/|v-v_*|$. Thus~\eqref{D:kernel-B} concludes the hard sphere, hard potentials
and Maxwell molecule models.

In the study of the renormalized solutions of the Boltzmann equation, Lions~\cite{Lio94} found that the gain operator
$Q^+(f,g)$ acts like a regularizing operator on each of its components when the other is frozen, i.e.
\begin{equation}\label{E:Lions's}
\|Q^+(f,g)\|_{H^1}\leq C\|g\|_{L^1}\|f\|_{L^2}\;,\;\|Q^+(f,g)\|_{H^1}\leq C\|g\|_{L^2}\|f\|_{L^1}
\end{equation}
under the assumption that $B(|z|,\theta)\in C^{\infty}_c(({\mathbb R}^3\backslash\{0\})\times (0,\frac{\pi}{2}))$. The exponent 
$1$ is due to that the compact support assumption. 
The definitions of  Sobolev spaces and the other function spaces are give in the notation subsection in the end of this section. 
The proof of~\eqref{E:Lions's} is based on a duality argument and the estimate of a Radon transform which 
relies the theory of Fourier integral operators (FIOs).
It is worthwhile to mention  that the regularity theory of the generalized Radon transform was studied in detail by Sogge and 
Stein~\cite{Sog85,Sog86,Sog90} at the end of the eighties.
Later Wennberg~\cite{Wen94} gave a simplified  proof of~\eqref{E:Lions's} by using the Carleman representation of $Q^+$ and 
classical Fourier transform. The estimates for full kernel without compactness assumption were given by 
Wennberg~\cite{Wen94}, Bouchut \& Desvillettes~\cite{BD98}, Lu~\cite{Lu98}, Mouhot \& Villani~\cite{Mou04} in the forms
different with~\eqref{E:Lions's}, see~~\cite{JC12} for more details. 
In~\cite{JC12}, the author proved that the estimates of the form~\eqref{E:Lions's} hold for the full kernels~\eqref{D:kernel-B}
for lower regularity. More precisely, with assumption~\eqref{D:kernel-B} and $\gamma>0$ we have
\begin{equation}\label{E:JC12}
\|Q^+(f,g)\|_{H^{\gamma-}}\leq C\|g\|_{L^1_\gamma}\|f\|_{L^2_\gamma}\;,\;\|Q^+(f,g)\|_{H^{\gamma-}}\leq C\|g\|_{L^2_\gamma}
\|f\|_{L^1_\gamma}
\end{equation}
where  $\gamma-$ means $\gamma-\varepsilon>0$ for arbitrary small $\varepsilon>0$ and 
thus the constant $C$ depends on $\varepsilon$. 

On the other hand, Gustafsson~\cite{Gus88} proved that $Q^+(f,g)$ can be regarded as 
a convolution operator, and he used this fact to prove uniform $L^p$ estimates for solutions of the space 
homogeneous Boltzmann equation. The estimates by Duduchava, Kirsch and Rjasanow~\cite{DKR06} , 
Alonso and Carneiro~\cite{AC10}, Alonso, Carneiro and Gamba~\cite{ACG10} and Alonso and Gamba~\cite{AG11}  are also of this type. 
Assume collision kernels 
\[
B(|z|,\theta)=|z|^{\lambda}b(\cos\theta)
\]  
with $b(\cos\theta)$ satisfies Grad's cut-off assumption.  Alonso, Carneiro and Gamba~\cite{ACG10} obtained 
that if $\lambda,\alpha\geq 0$ and $1\leq p,q,R\leq \infty$ with $1/p+1/q=1+1/R$, then
\begin{equation}\label{E:convolution-hard}
\|Q^+(f,g)\|_{L^{R}_{\alpha}}\leq C\|g\|_{L^p_{\alpha+\lambda}}\|f\|_{L^q_{\alpha+\lambda}}
\end{equation}
where  the explicit $C$ is given. 
For $-n<\lambda<0$ ( $n$ is the dimension of variable $z$ ) and
$1/p+1/q=1+\lambda/n+1/R$, they also obtained 
\begin{equation}\label{E:convolution-soft}
\|Q^+(f,g)\|_{L^R}\leq C\|g\|_{L^p}\|f\|_{L^q}.
\end{equation}

The main result of this paper is the new estimates in Theorem~\ref{T:main} 
below which not only improve~\eqref{E:JC12}, characterize both 
regularization and convolution properties of the gain term but also have the following features.
The regularizing exponent is sharp both in the $L^2$ based
inhomogeneous and homogeneous Sobolev spaces which is exact the exponent of the kinetic part of 
collision kernel. The functions in these estimates belong to a wider scope of (weighted) Lebesgue spaces   
than the previous regularizing estimates. For the estimates in homogeneous Sobolev spaces, we only 
need the functions lying in  Lebesgue spaces instead of weighted Lebesgue spaces, i..e., no loss
of weight occurs in this case. To our best knowledge, the homegeneous is not seen before in the literature for the estimate 
of the gain term with full kernel, note that the estimates of Lions requir compactness on relative velocity. More precisely, 
we have the following.
\begin{thm}\label{T:main}
Let $Q^+(f,g)$ be the operators defined by~\eqref{D:bi-Q-plus} with collision kernels~\eqref{D:kernel-B}.
For $1\leq p,q\leq 2 , 1/p+1/q=3/2$ we have estimates in homogeneous Sobolev spaces as
\begin{equation}\label{E:Gain-smooth}
\|Q^{+}(f,g)\|_{\dot{H}^{\gamma}}\leq C\|g\|_{L^p}\|f\|_{L^q}.
\end{equation}
and estimates in inhomogeneous Sobolev spaces as
\begin{equation}\label{E:Gain-smooth-in}
\|Q^{+}(f,g)\|_{{H}^{\gamma}}\leq C\|g\|_{L^p_\gamma}\|f\|_{L^q_\gamma},
\end{equation}
The constants $C$ depend on $p,q$ and $\gamma$.
\end{thm}

\begin{rem}
We note that the product of the angular function $b(\cos\theta)=\cos\theta$ and
$\sin\theta$, Jacobin of solid element, has the feature that it decays to $0$ when $\theta$ tends to 
$0$ or $\pi/2$ which is needed for our proof of estimates. Indeed we can get the same estimates~\eqref{E:Gain-smooth-in} and 
~\eqref{E:Gain-smooth} when $b(\cos\theta)=\cos\theta$ is
replaced by other  angular functions with the suitable decaying property near $\pi/2$, but we prefer to
use $b(\cos\theta)=\cos\theta$ for the simplicity of representation. 
\end{rem}

\begin{rem}
In many papers, the authors consider the collision operator operator defined on $S^2$ instead of physical one $S^2_+$ by 
extending evenly the range of $\theta$ to $[0,\pi]$.  This extension will not affect the results as we explained in the paragraph
right after~\eqref{E:a-g0-loc}. 
\end{rem}

The geometric argument in~\cite{JC12} can relax the compactness assumption in the angle part of the collision kernel but still need 
compactness on relative velocity. Thus the estimate ~\eqref{E:JC12} for the operator with full kernel must have loss of regularity 
as it comes from scaling argument and summation of estimates for different scale. Theorem~\ref{T:main} were
obtained through the new understanding on the role of relative velocity of the collision kernel. That is the  estimate~\eqref{E:Gain-smooth} 
disclose the fact that gain term trades the relative velocity for the regularity to the same exponent up to $1$, a fact which is vague in the 
estimate like~\eqref{E:JC12} or even~\eqref{E:Gain-smooth-in}.  Based on this, we can drop the compactness 
on relative velocity and integrate the geometric argument in~\cite{JC12}  to obtain the estimates  with features mentioned above, 
see more detail after Lemma~\ref{L:T-main-1} for technique explaination.

By the embedding theorem in homogeneous Sobolev space, see for example~\cite{BCD11}, we immediately obtain the following. 
\begin{cor} Let $Q^+(f,g)$ be the same as Theorem~\ref{T:main}. For $1\leq p,q\leq 2 , 1/p+1/q=3/2$  and $2\leq R$ 
satisfies $1/2=\gamma/3+1/R$, then we have 
\begin{equation}\label{E:Gain-integrability}
\|Q^{+}(f,g)\|_{L^R}\leq C\|g\|_{L^p}\|f\|_{L^q}.
\end{equation}
\end{cor}

It is interesting to note that the relation of $p,q,R,\gamma$ in above corollary can be restated as 
\begin{equation}
3/2=1/p+1/q=1+\gamma/3+1/R.
\end{equation}
This relation is the same as that for the soft potential given by~\eqref{E:convolution-soft} except that  the sum 
must be $3/2$. And the functions of estimates~\eqref{E:Gain-integrability} are lying in Lebesgue spaces as the 
estimates for the soft potentials.  
 The constraint $3/2$ is because our estimates in Theorem~\ref{T:main} is $L^2$ based. To get the estimates 
without such constraint, one needs the regularizing estimates for all $L^p,1\leq p\leq \infty$ based Sobolev spaces 
which will be done in our upcoming work. Then one can derive the unified estimates for collision operators
of all models and which are both scaling fit and have no loss of weight.  

It is interesting to note that regularizing exponent $\gamma$ is actually determined by the part of small relative velocity  
of kinetic factor of the collision kernels. If we modify~\eqref{D:kernel-B} to $(1+|v-v_*|)^{\gamma}\cos\theta$ or 
equivalently chop off small relative velocity smoothly, then the regularizing exponent is always $1$ independent 
of $\gamma$.  To give a precise statement, we let $\rho\in C^{\infty}(\mathbb{R}),\; 0\leq \rho \leq 1$ be supported in 
the open interval $(4,16)$
and satisfy 
\begin{equation}\label{D:partition-of-unity-rho}
1=\sum_{k\in\mathbb{Z}}\rho(2^{-k} r),\; r>0
\end{equation}
with property ${\rm supp}\;\rho(2^{j} r)\cap\; {\rm supp}\;\rho(2^k r)=\emptyset $ if $|j-k|\geq 2$.
Define $\mathbbm{s}(0)=0$ and 
\begin{equation}\label{D:mathbbm-s}
\mathbbm{s}(r)=\sum_{k\leq 0}\rho(2^{-k} r),\;r > 0.
\end{equation}
Then $\mathbbm{s}$ is a smooth positive function on $[\;0,\infty)$ satisfying 
$\mathbbm{s}(r)=1$ when $0\leq r<4$ and $\mathbbm{s}(r)=0$ when $r>16$.  Let 
$\overline{\mathbbm{s}}(r)=1-\mathbbm{s}(r), r\geq 0$. Then we 
split the collision kernel $B$ into small and large relative velocity parts, by
\[
B(z,\omega)=\mathbbm{s}(|z|)B(z,\omega)+\overline{\mathbbm{s}}(|z|)B(z,\omega)\eqdef 
B_{\mathbbm{s}}(z,\omega)+B_{\overline{\mathbbm{s}}}(z,\omega),
\]
and split $Q^+$ into two parts by plugging above into~\eqref{D:bi-Q-plus} and write
\begin{equation}\label{D:Q-velocity-split}
Q^+(f,g)=Q_\mathbbm{s}^+(f,g)+Q^+_{\overline{\mathbbm{s}}}(f,g). 
\end{equation} 
Then we have the following.
\begin{cor}\label{C:Q-L}
Let $Q^+(f,g),p,q$ be the same as Theorem~\ref{T:main}. Let $Q^+_{\overline{\mathbbm{s}}}$ be the $Q^+$
with small relative velocity being chopped off smoothly defined in~\eqref{D:Q-velocity-split}. Then we have
\begin{equation}\label{E:Gain-smooth-in-L}
\|Q^+_{\overline{\mathbbm{s}}}(f,g)\|_{{H}^{1}}\leq C\|g\|_{L^p_\gamma}\|f\|_{L^q_\gamma},
\end{equation}
and 
\begin{equation}\label{E:Gain-smooth-L}
\|Q^+_{\overline{\mathbbm{s}}}(f,g)\|_{\dot{H}^{1}}\leq C\|g\|_{L^p}\|f\|_{L^q}.
\end{equation}
The constants $C$ depend on $p,q$ and $\gamma$.
\end{cor}
The proof of above corollary will be given in Section~\ref{Proof of lemmas} since it is 
a consequence of the proof of the main Theorem.

{\bf Notations}\hfil\\
We set Japanese bracket  $\lr{v}=(1+|v|^2)^{1/2},\;v\in\mathbb{R}^3$, $(a\cdot b)=\sum_{i=1}^3 a_ib_i$  
the scalar product in $ \mathbb{R}^3$ and $\lr{f,g}=\int_{\mathbb{R}^3}f(x)g(x)dx$  the inner product in 
$L^2({\mathbb{R}^3}). $
The differential operator
$D^s,s\in\mathbb{R}$ is expressed through the Fourier transform:
\[
D^s f(x)=(2\pi)^{-3}\int_{\mathbb{R}^3} e^{ix\cdot\xi}|\xi|^s \widehat{f}(\xi)d\xi
\]
where the Fourier transform
$
\widehat{f}(\xi)=\int_{\mathbb{R}^3} e^{-ix\cdot\xi} f(x)dx.
$
The weighted Lebesgue, 
fractional homogeneous and inhomogeneous Sobolev spaces are denoted by
\[
\begin{split}
& \|f\|_{L^p_q}={\Big (}\int_{\mathbb{R}^3} |f(v)|^p\lr{v}^{pq} dv {\Big )}^{1/p}\;,\\
& \|f\|_{\dot{H}^{\alpha}}=\||\xi|^{\alpha}\widehat{f}(\xi)\|_{L^2}\;,\;
\|f\|_{{H}^{\alpha}}=\|\lr{\xi}^{\alpha}\widehat{f}(\xi)\|_{L^2}
\end{split}
\]

We use the multi-indices notation $\partial_x^{\alpha}=\partial_{x_1}^{\alpha_1}\cdots\partial_{x_d}^{\alpha_d}$.
By $\partial_x$ or by $\nabla_x$, we will denote the gradient.
A function $p(x,\xi)\in C^{\infty}(\mathbb{R}^d\times\mathbb{R}^d)$ satisfying 
\begin{equation}\label{D:symbol-definition}
 |\partial^{\alpha}_x \partial^{\beta}_{\xi}p(x,\xi)|
\leq C_{\alpha,\beta} \lr{\xi}^{m-|\beta|}
\end{equation} 
for any multi-indices $\alpha$ and $\beta$ is called a symbol of order $m$. The class of such function is  denoted by $S^{m}_{1,0}$.
We will see that the symbols $p(x,\xi)$ in this paper always enjoy the better decaying condition, i.e., 
\begin{equation}\label{D:symbol-definition-SG}
 |\partial^{\alpha}_x \partial^{\beta}_{\xi}p(x,\xi)|
\leq C_{\alpha,\beta}\lr{x}^{m_1-|\alpha|}\lr{\xi}^{m_2-|\beta|}
\end{equation} 
for any multi-indices $\alpha$ and $\beta$ and which is called a $SG$ symbol of order $(m_1,m_2)$, used by 
Cordes~\cite{Cord95} and Coriasco~\cite{Cori99}.
We  use  $SG^{m_1,m_2}$ to denote the set of such symbols.  For each $p(x,\xi)\in S_{1,0}^l$, the associate operators 
\[
P(x,D) h(x)=\int_{\mathbb{R}^n} e^{ix\cdot\xi} p(x,\xi) \widehat{h}(\xi)d\xi
\]
is called a pseudodifferential operator of order $l$.  The standard notation $S^{-\infty}_{1,0}=\cap_{\;l} S^{l}_{1,0},l\in\mathbb{Z}$
is also used. If $p(x,\xi)\in S^{-\infty}_{1,0}$, it is called a symbol of the smooth operator.
The operator 
\[
Sf(x)=\int_{\mathbb{R}^3} e^{i\psi(x,\xi)} a(x,\xi) \widehat{f}(\xi)d\xi
\]
with symbol $a(x,\xi)\in S^l_{1,0}$  and the phase function $\psi(x,\xi)$ satisfies  
non-degeneracy condition is called a Fourier integral operator of order $l$.
We say a phase function $\psi(x,\xi)$
satisfies the non-degeneracy condition if there is a constant $c>0$ such that 
\begin{equation}\label{D:nondegeneracy}
|\det  {\Big [} \partial_{x}\partial_{\xi} \psi(x,\xi) {\Big ]} | =|\det  {\Big [} \partial^2\psi(x,\xi)/\partial_{x_i}\partial_{\xi_j}  {\Big ]} |\geq c>0
\end{equation} 
for all $(x,\xi)\in{\rm supp}\;a(x,\xi)$.

\section{Reduction and Almost Orthogonality}

Inspired by Lions~\cite{Lio94}, we found that the proof of the Theorem~\ref{T:main} relies on 
the understanding of the following Radon transform  
\begin{equation}\label{Def:T}
\mathbb{T}h(x)=|x|^{\gamma}\int_{\omega\in S^2_+} b(\cos\theta)h(x-(x\cdot\omega)\omega) d\Omega(\omega), 
\;0\leq\gamma\leq 1
\end{equation}
with
$\cos\theta={(x\cdot\omega)}/{|x|}, x\neq 0$, $x=|x|(0,0,1)$ and 
$\omega=(\cos\varphi\sin\theta,\sin\varphi\sin\theta,\cos\theta)$, $0\leq\theta\leq \pi/2 $. 
We use the notation $\|F(v,v_*)\|_{L^p(v)}$ to denote the $L^p$ norm of 
$F(v,v_*)$ with respect to variable $v$ where variable $v_*$ is regarded as the parameter.  
We will see soon that the estimate~\eqref{E:Gain-smooth} is the consequence of the following 
Lemma which concerning the estimates of $\mathbb{T}$.
\begin{lem}\label{L:T-v-v_*-ineq}
Let $\mathbb{T}$ be the operator defined by~\eqref{Def:T} and $\tau_m$ be the translation 
operator $\tau_m h(\cdot)=h(\cdot+m)$. We have
\begin{equation}\label{E:T-v-ineq}  
\sup\limits_{v_*}\|(\tau_{-v_*}\circ \mathbb{T}\circ\tau_{v_*} )h(v)\|_{L^2(v)}
\leq C \|h\|_{\dot{H}^{-\gamma}}
\end{equation}
and
\begin{equation}\label{E:T-v_*-ineq} 
\sup\limits_{v}\|(\tau_{-v_*}\circ \mathbb{T}\circ\tau_{v_*} )h(v)\|_{L^2(v_*)}
\leq C \|h\|_{\dot{H}^{-\gamma}}.
\end{equation}
\end{lem}
Please note that the left hand side of ~\eqref{E:T-v-ineq} is an integration of
variable $v$ with parameter $v_*$ while the roles of $v$ and $v_*$ are switched 
in ~\eqref{E:T-v_*-ineq}. 
On the other hand an unified perspective is provided in the proof of  these two estimates
in section~\ref{Proof of lemmas} with the aid of the structural understanding of $\mathbb{T}$
in  Lemma~\ref{L:T-main} below.

To find the key estimates for the proof of ~\eqref{E:Gain-smooth-in}, note
that the estimates~\eqref{E:Gain-smooth-in} follow from 
\begin{equation}\label{E:Gain-in-c}
\|Q^{+}_\mathbbm{s}(f,g)\|_{{H}^{\gamma}}\leq C\|g\|_{L^p}\|f\|_{L^q},
\end{equation}
and 
\begin{equation}\label{E:Gain-in-oc}
\|Q^{+}_{\overline{\mathbbm{s}}}(f,g)\|_{{H}^{\gamma}}\leq C\|g\|_{L^p_\gamma}\|f\|_{L^q_\gamma}.
\end{equation}
We define $|x|^\gamma_\mathbbm{s}=\mathbbm{s}(|x|)|x|^\gamma$ and follow~\eqref{Def:T} to define 
\begin{equation}\label{Def:T-c}
\mathbb{T}_\mathbbm{s}h(x)=|x|^\gamma_\mathbbm{s}\int_{\omega\in S^2_+} b(\cos\theta)h(x-(x\cdot\omega)\omega) d\Omega(\omega),
\end{equation}
and
\begin{equation}\label{Def:H-s-oc}
H_{\overline{\mathbbm{s}}}(v,v_*)=\int_{S^2_+} \frac{1}{\lr{v}^\gamma\lr{v_*}^\gamma}B_{\overline{\mathbbm{s}}}
(v-v_*,\omega)h(v')d\Omega(\omega)
\end{equation}
whose structure is similar to that of
$\tau_{-v_*}\circ \mathbb{T}\circ\tau_{v_*}$. We will see soon that the estimates~\eqref{E:Gain-in-c} 
and~\eqref{E:Gain-in-oc} follow the Lemma below. 
\begin{lem}\label{L:H-v-v_*-ineq}
Let $\mathbb{T}_\mathbbm{s}$ and $H_{\overline{\mathbbm{s}}}$ be the defined by ~\eqref{Def:T-c} and~\eqref{Def:H-s-oc}. We have
\begin{equation}\label{E:T-c-v-ineq}  
\sup\limits_{v_*}\|(\tau_{-v_*}\circ \mathbb{T}_\mathbbm{s}\circ\tau_{v_*} )h(v)\|_{L^2(v)}\leq C \|h\|_{{H}^{-\gamma}},
\end{equation}
\begin{equation}\label{E:T-c-v_*-ineq} 
\sup\limits_{v}\|(\tau_{-v_*}\circ \mathbb{T}_\mathbbm{s}\circ\tau_{v_*} )h(v)\|_{L^2(v_*)}\leq C \|h\|_{{H}^{-\gamma}},
\end{equation}
\begin{equation}\label{E:H-v-ineq}  
\sup\limits_{v_*}\|H_{\overline{\mathbbm{s}}}(v,v_*)\|_{L^2(v)}\leq C \|h\|_{{H}^{-\gamma}},
\end{equation}
and
\begin{equation}\label{E:H-v_*-ineq} 
\sup\limits_{v}\|H_{\overline{\mathbbm{s}}}(v,v_*)\|_{L^2(v_*)}\leq C \|h\|_{{H}^{-\gamma}}.
\end{equation}
\end{lem}

Assume Lemma~\ref{L:T-v-v_*-ineq} and Lemma~\ref{L:H-v-v_*-ineq} hold temporarily, whose proofs
are postponed to the Section~\ref{Proof of lemmas}, we can 
prove the Theorem~\ref{T:main}.
\begin{proof}[Proof of Theorem~\ref{T:main}]
First we consider the proof of the estimate~\eqref{E:Gain-smooth}. 
By the duality of homogeneous Sobolev spaces, see for example~\cite{BCD11}, we need to show that
\[
{\big |}\lr{Q^+(f,g),h}{\big |}\leq C \|f\|_{L^q}\|g\|_{L^p}\|h\|_{\dot{H}^{-\gamma}}
\]
holds for any $h\in \dot{H}^{-\gamma}$.
Using change of variables and H\"{o}lder inequality
\[
\begin{split}
&{\big |}\lr{Q^+(f,g),h}{\big |}\\
&={\big |}\iiint f(v)g(v_*)B(v-v_*,\omega)h(v')d\Omega(\omega) dv_* dv{\big |}\\
&\leq \|f(v)\|_{L^{q}}\|Q^{+t}_g(h)(v)\|_{L^{q'}}
\end{split}
\]
where $q'$ is the conjugate exponent of $q$ and 
\[
Q^{+,t}_g(h)(v)=\iint g(v_*)B(v-v_*,\omega)
h(v')d\Omega(\omega) dv_*.
\]
Thus it is reduced to proving that 
\begin{equation}\label{E:Q-adjoint-ineq1}
\|Q^{+,t}_g(h)(v)\|_{L^{q'}}\leq C\|g\|_{L^p}\|h\|_{\dot{H}^{-\gamma}}.
\end{equation}
where $1/2+1/q'=1/p,\;1\leq p,q\leq 2$. 

Defining  the translation operator
\begin{equation}\label{D:translation}
\tau_{m}h(\cdot)=h(\cdot+m),
\end{equation}
and following Lions~\cite{Lio94}, we rewrite 
\[
Q^{+,t}_{g}(h)(v)=\int g(v_*)(\tau_{-v_*}\circ \mathbb{T}\circ\tau_{v_*} )h(v) dv_*.
\]
Define  $H(v,v_*)=(\tau_{-v_*}\circ \mathbb{T}\circ\tau_{v_*})h(v)$. By $1/q'+(q'-2)/(2q')+(q'-p)/(pq')=1$
and H\"{o}lder inequality, we have 
\[
\begin{split}
&{\big |}Q^{+t}_g(h)(v){\big |}\leq \int {\big |}H(v,v_*){\big |}{\big |}g(v_*){\big |}dv_*\\
&=\int {\Big(} {\big |} H(v,v_*){\big |}^2{\big |}g(v_*){\big |}^p {\Big )}^{1/q'}{\big |}H(v,v_*){\big |}^{(q'-2)/q'}
{\big |}g(v_*){\big |}^{(q'-p)/q'} dv_*\\
&\leq \|( |H(v,v_*)|^2|g(v_*)|^p)^{1/q'}\|_{L^{q'}(v_*)}\times \\
&{\hskip 2cm}\||H(v,v_*)|^{(q'-2)/q'}\|_{L^{\frac{2q'}{q'-2}}(v_*)}
\||g(v_*)|^{(q'-p)/q'}\|_{L^{\frac{pq'}{q'-p}}(v_*)}.
\end{split}
\]
Here
\begin{equation}\label{E:G-v_*}
\||g(v_*)|^{(q'-p)/q'}\|_{L^{\frac{pq'}{q'-p}}(v_*)}=(\|g(v_*)\|_{L^p(v_*)})^{\frac{q'-p}{q'}},
\end{equation}
\begin{equation}\label{E:H-v-v_*-1}
\||H(v,v_*)|^{(q'-2)/q'}\|_{L^{\frac{2q'}{q'-2}}(v_*)}=(\|(\tau_{-v_*}\circ \mathbb{T}\circ\tau_{v_*} )h(v)\|_{L^2(v_*)})^{\frac{q'-2}{q'}},
\end{equation}
and 
\begin{equation}\label{E:HG-v-v_*}
\begin{split}
&\|( |H(v,v_*)|^2|g(v_*)|^p)^{1/q'}\|_{L^{q'}(v_*)}\\
&=(\int {\big |}g(v_*){\big |}^p {\big |}(\tau_{-v_*}\circ \mathbb{T}\circ\tau_{v_*} )h(v) {\big |}^2 dv_*)^{\frac{1}{q'}}.
\end{split}
\end{equation}
Applying~\eqref{E:T-v_*-ineq} of Lemma~\ref{L:T-v-v_*-ineq}, we 
have 
\begin{equation}\label{E:H-v-v_*-2}
~\eqref{E:H-v-v_*-1}\leq C (\|h\|_{\dot{H}^{-\gamma}})^{\frac{q'-2}{q'}}
\end{equation}
where $C$ is independent of $v$.
Combine~\eqref{E:G-v_*},~\eqref{E:H-v-v_*-2} and
~\eqref{E:HG-v-v_*}, we have
\begin{equation}\label{E:combine-1}
\begin{split}
&\|Q^{+,t}_g(h)(v)\|_{L^{q'}(v)}^{q'} \\
&\leq C (\|g(v_*)\|_{L^p(v_*)})^{(q'-p)}(\|h\|_{\dot{H}^{-\gamma}(v_*)})^{(q'-2)}\times\\
& {\hskip 2cm}\int\int {\big |}g(v_*){\big |}^p {\big |}(\tau_{-v_*}\circ \mathbb{T}\circ\tau_{v_*} )h(v) {\big |}^2 dv_* dv\\
&\leq C (\|g(v_*)\|_{L^p(v_*)})^{(q'-p)}(\|h\|_{\dot{H}^{-\gamma}(v_*)})^{(q'-2)}\times\\
& {\hskip 2cm} \int |g(v_*)|^p \sup\limits_{v_*}\|
(\tau_{-v_*}\circ \mathbb{T}\circ\tau_{v_*} )h(v)\|_{L^2_{v}}^2  dv_*. \\
\end{split}
\end{equation}
Then we conclude~\eqref{E:Q-adjoint-ineq1} by  applying ~\eqref{E:T-v-ineq} of 
Lemma~\ref{L:T-v-v_*-ineq} to the last term of ~\eqref{E:combine-1}. The proof of~\eqref{E:Gain-smooth}
is complete.

As we mentioned before the estimate~\eqref{E:Gain-smooth-in} comes from ~\eqref{E:Gain-in-c} and
~\eqref{E:Gain-in-oc}.
 Clearly the above argument also indicates that the estimate~\eqref{E:Gain-in-c} comes from  
~\eqref{E:T-c-v-ineq} and~\eqref{E:T-c-v_*-ineq}. 
Next we prove the estimate~\eqref{E:Gain-in-oc} is the consequence of~\eqref{E:H-v-ineq} and 
~\eqref{E:H-v_*-ineq}. 
By the duality, it is equivalent show that for any
$h\in H^{-\gamma}$ we have  
\[
{\big |}\lr{Q^+_{\overline{\mathbbm{s}}}(f,g),h}{\big |}\leq C \|f\|_{L^q_{\gamma}}\|g\|_{L^p_{\gamma}}\|h\|_{{H}^{-\gamma}},
\]
where $L^p_{\gamma}$ is weighted Lebesgue space, see notations in the end of first section.
Using change of variables and H\"{o}lder inequality,
\[
\begin{split}
&{\big |}\lr{Q^+_{\overline{\mathbbm{s}}}(f,g),h}{\big |}\\
&={\big |}\iiint \lr{v}^{\gamma}f(v)\lr{v_*}^{\gamma}g(v_*)\frac{1}{\lr{v}^{\gamma}\lr{v_*}^{\gamma}}B_{\overline{\mathbbm{s}}}(v-v_*,\omega)
h(v')d\Omega(\omega) dv_* dv{\big |}\\
&\leq \|f(v)\|_{L^{q}_{\gamma}}\|Q^{+,t}_{_{\overline{\mathbbm{s}}},g}(h)(v)\|_{L^{q'}}
\end{split}
\]
where $q'$ is the conjugate exponent of $q$ and 
\[
Q^{+,t}_{_{\overline{\mathbbm{s}}},g}(h)(v)=\iint \lr{v_*}^{\gamma}g(v_*)\frac{1}{\lr{v}^{\gamma}\lr{v_*}^{\gamma}}B_{\overline{\mathbbm{s}}}(v-v_*,\omega)
h(v')d\Omega(\omega) dv_*.
\]
It is reduced to proving that 
\begin{equation}\label{E:Q-adjoint-ineq1-in}
\|Q^{+,t}_{_{\overline{\mathbbm{s}}},g}(h)(v)\|_{L^{q'}}\leq C\|g\|_{L^p_{\gamma}}\|h\|_{{H}^{-\gamma}}
\end{equation}
where $1/2+1/q'=1/p,\;1\leq p,q\leq 2$. 
Recalling 
\[
H_{\overline{\mathbbm{s}}}(v,v_*)=\int_{S^2} \frac{1}{\lr{v}^{\gamma}\lr{v_*}^{\gamma}}B_{\overline{\mathbbm{s}}}(v-v_*,\omega)h(v')d\Omega(\omega),
\]
thus we rewrite 
\[
Q^{+,t}_{_{\overline{\mathbbm{s}}},g}(h)(v)=\int \lr{v_*}^{\gamma}g(v_*)H_{\overline{\mathbbm{s}}}(v,v_*) dv_*.
\]
Applying the H\"{o}lder inequality as the previous argument, we see that the result follows from
\[
\begin{split}
& \sup_{v_*}\|H_{\overline{\mathbbm{s}}}(v,v_*)\|_{L^2(v)}\leq C\|h\|_{H^{-\gamma}}\\
& \sup_{v}\|H_{\overline{\mathbbm{s}}}(v,v_*)\|_{L^2(v_*)}\leq C\|h\|_{H^{-\gamma}}\\
\end{split}
\]
which are the the last two estimates of Lemma~\ref{L:H-v-v_*-ineq}.
\end{proof}

The proof of Lemma~\ref{L:T-v-v_*-ineq} and Lemma~\ref{L:H-v-v_*-ineq} depend on the understanding of the 
operator $\mathbb{T}$ defined in~\eqref{Def:T}, while it needs 
a lot of effort.  
 For our purpose, we use inverse Fourier transform to rewrite
\begin{equation}\label{D:T-Fourier-1}
\begin{split}
\mathbb{T}h(x)
&=|x|^{\gamma}\int_{\omega\in S^2_+} b(\cos\theta)h(x-(x\cdot\omega)\omega)d\Omega(\omega)\\
&=(2\pi)^{-3}\int_{\mathbb{R}^3} e^{i{x\cdot\xi}} \mathbb{A}(x,\xi) \widehat{h}(\xi) d\xi
\end{split}
\end{equation}
where 
\[
\mathbb{A}(x,\xi)=|x|^{\gamma}\int_{\omega\in S^2_+} e^{-i(x\cdot\omega)(\xi\cdot\omega)} b(\cos\theta) d\Omega(\omega).
\]
The operator $\mathbb{T}$ originally defined on $S^2_+$ is turned to be an operator 
defined on $\mathbb{R}^3$ and whose property  is thus condensed in function 
$\mathbb{A}(x,\xi)$. The function $\mathbb{A}(x,\xi)$ is rather complicated and 
has differently properties on different portions of the phase space $(x,\xi)$ according to the scale of $x\cdot \xi$. 
However the following estimate holds.
\begin{lem}\label{L:T-main-1}
The operator $\mathbb{T}$ defined 
in~\eqref{D:T-Fourier-1} satisfies
\begin{equation}\label{E:T-main-1}
\|\mathbb{T}h\|_{L^2}\leq C \|h\|_{\dot{H}^{-\gamma}}.
\end{equation}
\end{lem}

As we mentioned in the paragraph after Theorem~\ref{T:main}, the geometric argument in~\cite{JC12} can only handle the 
case when $|x|$ in~\eqref{D:T-Fourier-1} being restricted to a compact set, though $b(\cos\theta)$ can be full range.  
With the compactness assumption, $|x|$ is basically regarded as a constant in~\cite{Lio94,JC12}. 
The key observation here is that the operator $\mathbb{T}$ trades $x$ for $\xi$ and thus it is more natural to consider the 
operator, say $T$, with kernel
\[
|x|^{\gamma}|\xi|^{\gamma}\int_{\omega\in S^2_+} e^{-i(x\cdot\omega)(\xi\cdot\omega)} b(\cos\theta) d\Omega(\omega)
\]
and prove a $L^2$ to $L^2$ estimate as we will do later.  Please note the symmetry of $x$ and $\xi$ in this kernel, thus 
the appearance of SG symbol (see the notation section) is natural and which can help us to sum up the dyadic partitions. 
On the other hand, we will see that the operator $\mathbb{T}$ behaves like a Fourier integral operator (FIO) only on some 
portion of phase space $(x,\xi)$ while its non-degeneracy condition is not uniformly. Also on small portion of phase space, 
$\mathbb{T}$ degenerate and the concept of symbol does not apply. Hence the standard theory of FIO does not work here.
To deal this complexity,  we  integrate  the  geometric argument in~\cite{JC12} with the new idea here to a rather 
different argument which can fit all scale of $x\cdot\xi$.

We remark that  in the study of the non-cutoff Boltzmann collision operator,   Alexandre and Villani~\cite{AV02} 
considered a parallel operator of the from
\begin{equation}\label{D:psd}
|x|^{\alpha} \int_{\omega\in S^2} (\cos\theta)^{-(1+2s)} [h(x-(x\cdot\omega)\omega))-h(x)] d\Omega(\omega), 0<s<2.
\end{equation}
Unlike the operator~\eqref{D:T-Fourier-1}, the main part of the~\eqref{D:psd}  is the pseudodifferential operator, see 
for example~\cite{Ale09}. The mechanisms in~\eqref{D:T-Fourier-1} and ~\eqref{D:psd} are also different. 
The angular singularity of~\eqref{D:psd} demands loss of regularity and weight, see for example the estimates 
in~\cite{AMUXY09,He16,JL19}.

Before go the the proof of Lemma~\ref{L:T-main-1},  
we introduce a preparation lemma which employs the almost orthogonality argument, see for example~\cite{Ste93}, 
to show the $L^2$ boundedness of a particular FIO induced by $\mathbb{A}(x,\xi)$.

We consider the cone 
defined in the phase space $(x,\xi)\in\{\mathbb{R}^3-\{0\}\}\times\{\mathbb{R}^3-\{0\}\}$ by
\begin{equation}\label{D:cone-product}
\begin{split}
& \{(x,\xi)|\; \frac{\pi}{8}<\arccos(\frac{x\cdot\xi}{|x||\xi|}) <\pi-\frac{\pi}{8} \} \\
&\eqdef\;\Gamma_x\times\Gamma_{\xi}.
\end{split}
\end{equation}
For each fixed $x$,  we use $x\times \Gamma_{\xi}$ to denote the cone whose element $\xi\in\mathbb{R}^3$ 
satisfies~\eqref{D:cone-product}. The notation $\Gamma_{\xi}$ means the cone $x\times\Gamma_{\xi}$ 
where the vector $x$ is not specified. Also $\Gamma_{x}\times x,\Gamma_x$ are defined likewise.  
To estimate the FIO whose amplitude function is defined on $\Gamma_x\times\Gamma_{\xi}$, we need a dyadic
partition of unity on $\mathbb{R}^3-\{0\}$. 
For $x\in\mathbb{R}^3$ and $k\in\mathbb{Z}$ we define  $\chi_{k}(x)=\rho(2^{-k}|x|)$ where $\rho$ is defined 
in~\eqref{D:partition-of-unity-rho}, then we have dyadic partition of unity
\begin{equation}\label{D:chi-k}
1=\sum_{k\in \mathbb{Z}} \chi_{k}(x) ,\; x\neq  0.
\end{equation}
We also need the definition
\begin{equation}\label{E:Q}
\mathcal{Q}=\max \{ \int_{\Gamma_{x}}\chi_{0}(x)dx\;,\; 
\int_{ \Gamma_{x}}\chi_{0}(x)\chi_{1}(x)dx \}.
\end{equation}
The lemma we need is the following.
\begin{lem}\label{lemma:Fourier_integral}
Let $F$ be  defined by
\begin{equation}\label{D:T-RS-lemma}
Fu(x)=\int_{{\mathbb R}^3}e^{i\psi(x,\xi)} p(x,\xi)\;u(\xi)d\xi
\end{equation} 
where $\psi(x,\xi)$ is homogeneous of degree $1$ in $x$ and $\xi$, $p(x,\xi)\in C^{\infty}({\mathbb R}^3_x
\times{\mathbb R}^3_{\xi})$ and ${\rm supp}\; p(x,\xi)\subset\Gamma_x\times\Gamma_\xi\cap\{|x|>8,|\xi|>8\}$
satisfies
\begin{equation}\label{E:symbol-derivative-bound}
|\partial^{\alpha}_x\partial^{\beta}_{\xi} p(x,\xi)|\leq C_{\alpha\beta}|x|^{-|\alpha|}|\xi|^{-|\beta|},
\end{equation}
for $|\alpha|,|\beta|\leq 4$. Let 
\begin{equation}\label{D:P}
\mathcal{P}=\sup\limits_{|\alpha|,|\beta|\leq 4}\|\partial^{\alpha}_x\partial^{\beta}_{\xi}  p(x,\xi)\|_{L^{\infty}
(\mathbb{R}^3_x\times\mathbb{R}^3_{\xi})}.
\end{equation}
When $(x,\xi)\in {\rm supp}\;p(x,\xi)$, the phase function $\psi(x,\xi)$ satisfies 
\begin{equation}\label{E:non_degenerate0}
0<C_1<|\det {\Big [} \partial_{x}\partial_{\xi} \psi(x,\xi) {\Big ]}|<C_2
\end{equation}  and
\begin{equation}\label{E:phase-upper-bound}
|\partial^{\alpha}_x\partial_{\xi}\psi(x,\xi)|\leq C_{\alpha}\;,\;|\partial_x\partial^{\beta}_{\xi}\psi(x,\xi)|
\leq C_{\beta}\;,\;1\leq|\alpha|,|\beta|\leq 5 %
\end{equation}
Then $F$ is $L^2$ bounded and satisfies 
\begin{equation}\label{E:L-2-bound}
\|F\|_{L^2\rightarrow L^2}\leq \mathcal{C} \mathcal{Q}\mathcal{P}
\end{equation}
where $\mathcal{C}$ depends on the constants in~\eqref{E:non_degenerate0} and~\eqref{E:phase-upper-bound} and $\mathcal{Q}$
is defined in~\eqref{E:Q}.
\end{lem}
\begin{proof}
Note $\sum_{k=0}^{\infty}\chi_{k}(z)=1$  when $|z|> 8$. Therefore we can decompose $F$ as
\[
F=\sum_{(j,l)\in\mathbb{N}\cup\{0\}\times\mathbb{N}\cup\{0\}} F_{(j,l)}
\]
where 
\[
F_{(j,l)}u(x)=\chi_{j}(x)\int_{{\mathbb R}^3}e^{i\psi(x,\xi)} \chi_{l}(\xi) p(x,\xi)u(\xi)d\xi
\]
The adjoint of $F_{j,l}$, denoted by $F^{*}_{j,l}$ is
\[
F^{*}_{(j,l)}v(\xi)=\chi_l(\xi)\int e^{-i\psi(y,\xi)}\chi_j(y)\overline{p(y,\xi)}v(y)dy.
\]
Then we have
$$F_{(j,l)}F^{*}_{(k,m)}u(x)=\int K_{(j,l),(k,m)}(x,y)u(y)dy,$$
where
\begin{equation}\label{E:K}
\begin{split}
& K_{(j,l),(k,m)}(x,y)\\
&=\chi_j(x)\chi_k(y)\int e^{i(\psi(x,\xi)-\psi(y,\xi))}\chi_l(\xi)\chi_m(\xi)p(x,\xi)\overline{p(y,\xi)}
d\xi.
\end{split}
\end{equation}
Since ${\rm supp}\;\chi_l(\xi)\cap{\rm supp}\;\chi_m(\xi)=\emptyset$ when $|l-m|\geq 2$, we only have to consider $l=m-1,m$ or $m+1$. 
Without loss of generality, we assume $l=m+1$ and $j\geq k$. 
First we consider the sub-case $j\geq k+3$.
Let 
\[
\ti{x}=2^{-k}x=\tau^{-1}_1 x, \ti{y}=2^{-k}y=\tau^{-1}_1 y, \ti{\xi}=
2^{-m}\xi=\tau_2^{-1}\xi.
\]
 Since $\psi$ is homogeneous of degree $1$ in first and second variables, we have
\[
\begin{split}
&K_{(j,l)(k,m)}(x,y)\\
&=\tau_2^{3}\chi_{j-k}(\ti{x})\chi_0(\ti{y})\int e^{i\tau_1\tau_2(\psi(\ti{x},\ti{\xi})-\psi(\ti{y},\ti{\xi}))}\chi_0(\ti{\xi})
\chi_1(\ti{\xi})p(\tau_1\ti{x},\tau_2\ti{\xi})\overline{p(\tau_1\ti{y},\tau_2\ti{\xi})}d\ti{\xi} \\
&\eqdef\tau_2^{3} \;\mathcal{K}_{(j,l)(k,m)}(\ti{x},\ti{y}).
\end{split}
\]
Hence 
\begin{equation}\label{D:F-j-l-k-m}
F_{(j,l)}F^{*}_{(k,m)} u(\tau_1\ti{x}) =\int \tau_2^{3}\mathcal{K}_{(j,l)(k,m)}(\ti{x},\ti{y}) u(\tau_1 \ti{y})
 \tau_1^{3} d\ti{y}.
\end{equation}
Define the operator 
\begin{equation}\label{E:L-change}
L_{\tau_1\tau_2}=\frac{1}{i\tau_1\tau_2}\frac{\partial_{\ti{\xi}}\psi(\ti{x},\ti{\xi})-\partial_{\ti{\xi}}\psi(\ti{y},\ti{\xi})}
{|\partial_{\ti{\xi}}\psi(\ti{x},{\ti{\xi}})-\partial_{\ti{\xi}}\psi(\ti{y},{\ti{\xi}})|^2}\cdot\partial_{\ti{\xi}}
\end{equation}
and observe that
$$L_{\tau_1\tau_2}(e^{{i}{\tau_1\tau_2}(\partial_{\ti{\xi}}\psi(\ti{x},\ti{\xi})-\partial_{\ti{\xi}}\psi(\ti{y},\ti{\xi}))})
=e^{{i}{\tau_1\tau_2}(\partial_{\ti{\xi}}\psi(\ti{x},\ti{\xi})-\partial_{\ti{\xi}}\psi(\ti{y},\ti{\xi})}.$$
Integration by parts yields
$$
\begin{aligned}
&\int e^{i\tau_1\tau_2(\psi(\ti{x},\ti{\xi})-\psi(\ti{y},\ti{\xi}))}\chi_0(\ti{\xi})
\chi_1(\ti{\xi})p(\tau_1\ti{x},\tau_2\ti{\xi})\overline{p(\tau_1\ti{y},\tau_2\ti{\xi})}d\ti{\xi}\\
&=\int e^{i\tau_1\tau_2(\psi(\ti{x},\ti{\xi})-\psi(\ti{y},\ti{\xi}))}(L^*_{\tau_1\tau_2})^4 (\chi_0(\ti{\xi})
\chi_1(\ti{\xi})p(\tau_1\ti{x},\tau_2\ti{\xi})\overline{p(\tau_1\ti{y},\tau_2\ti{\xi})})d\ti{\xi},
\end{aligned}
$$
where $L_{\tau_1\tau_2}^*$ is the transpose of $L_{\tau_1\tau_2}$.
Since $p(x,\xi)$ satisfies~\eqref{E:symbol-derivative-bound}, we have
\begin{equation}\label{D:P-1}
\sup_{|\alpha|,|\beta|\leq 4}\|\partial^{\alpha}_{\ti{x}}\partial^{\beta}_{\ti{\xi}}  p(\tau_1\ti{x},\tau_2\ti{\xi}
)\|_{L^{\infty}(\mathbb{R}^3_{\ti{x}}\times\mathbb{R}^3_{\ti{\xi}})}\leq \mathcal{P}.
\end{equation}
From the assumption~\eqref{E:non_degenerate0} we obtain
\begin{equation}\label{E:non_degenerate1}
|\partial_{\ti{\xi}}\psi(\ti{x},\ti{\xi})-\partial_{\ti{\xi}}\psi(\ti{y},\ti{\xi})|\geq C|\ti{x}-\ti{y}|
\end{equation}
($C$ depends on $C_1$, see~\cite{Ste93} P.397 for obtaining~\eqref{E:non_degenerate1} from ~\eqref{E:non_degenerate0})
and
\[
|\partial^{\beta}_{\ti{\xi}}\psi(\ti{x},\ti{\xi})-\partial^{\beta}_{\ti{\xi}}\psi(\ti{y},\ti{\xi})|\leq C_{\beta}|\ti{x}-\ti{y}|
\]
for $1\leq|\beta|\leq 5$ since phase function $\psi$ is homogeneous of degree $1$ in first variable. Note $\ti{x}\in\;{\rm supp}\chi_{j-k},\ti{y}\in\;{\rm supp}\chi_{0}$ and 
$j\geq k+3$. Hence we have
\begin{equation}\label{E:decay_kernel}
\begin{split}
|\mathcal{K}_{(j,l)(k,m)}(\ti{x},\ti{y})| \leq & C(\tau_1\tau_2)^{-4} \mathcal{P}^2\;\frac{\chi_{j-k}(\ti{x})
\chi_0(\ti{y})}{1+|\ti{x}-\ti{y}|^4}\\
&{\hskip 2cm }\times\int_{\mathbb{R}^3}\chi_{0}|_{y\times\Gamma_{\xi}}(\ti{\xi})\chi_{1}|_{x\times\Gamma_{\xi}}(\ti{\xi})d\ti{\xi} 
\end{split}
\end{equation}
where $C$ depends on the constants of~\eqref{E:non_degenerate0} and~\eqref{E:phase-upper-bound} 
and $\chi_{0}|_{y\times\Gamma_{\xi}},\chi_{1}|_{x\times\Gamma_{\xi}}$ are functions $\chi_0(\xi),\chi_1(\xi)$ restricted 
to cones $\Gamma_{\xi}$ determined by $x,y$ respectively. 
We note that 
\begin{equation}\label{E:cone-span}
\int_{\mathbb{R}^3}\chi_{0}|_{y\times\Gamma_{\xi}}(\ti{\xi})\chi_{1}|_{x\times\Gamma_{\xi}}(\ti{\xi})d\ti{\xi} 
\end{equation}
attains its maximum when $x=cy$ for $c>0$. The
maximum of~\eqref{E:cone-span} can be written as
\[
\int_{x\times \Gamma_{\ti{\xi}}}\chi_{0}(\ti{\xi})\chi_{1}(\ti{\xi})d\ti{\xi} 
\]
and note that its value  depends on the the span of cone $\Gamma_{\xi}$. The corresponding part for the
case $l=m-1$ clearly has the same maximum as above.    
For $l=m$, the corresponding maximum will be 
\[
\int_{x\times \Gamma_{\ti{\xi}}}\chi_{0}(\ti{\xi})\chi_{0}(\ti{\xi})d\ti{\xi}. 
\]
By definition of $\Gamma_x\times\Gamma_{\xi}$ and the fact $0\leq \chi_{0}(\cdot)\leq 1$,
\[
\max \{ \int_{x\times \Gamma_{\ti{\xi}}}\chi_{0}(\ti{\xi})\chi_{0}(\ti{\xi})d\ti{\xi}\;,\; 
\int_{x\times \Gamma_{\ti{\xi}}}\chi_{0}(\ti{\xi})\chi_{1}(\ti{\xi})d\ti{\xi} \}\leq \mathcal{Q}.
\]
We also note that~\eqref{E:cone-span} non-vanishes only when the angle spanned by $x,y$
is in a suitable range determined by the definition~\eqref{D:cone-product}.  
From above and $\ti{x}\in {\rm supp}\;\chi_{j-k}\;,\;\ti{y}\in {\rm supp}\;\chi_0$, we have 
\begin{equation}\label{E:mathcal-K-bound}
\begin{split}
& \sup_{\ti{x}}\int|\tau_2^{3} \mathcal{K}_{(j,l)(k,m)}(\ti{x},\ti{y})\tau_1^{3}|d\ti{y} \leq C 2^{-(4j-3k)}2^{-m} \mathcal{Q} \mathcal{P}^2
\int_{ \Gamma_{\ti{y}}} \chi_0(\ti{y})d\ti{y} \\
&{\hskip 5.1cm}\leq C 2^{-j}2^{-m} \mathcal{Q}^2 \mathcal{P}^2\\
&{\hskip 5.1cm}= C 2^{-(j-k)}2^{-k}2^{-m} \mathcal{Q}^2 \mathcal{P}^2,\\
& \sup_{\ti{y}}\int|\tau_2^{3} \mathcal{K}_{(j,l)(k,m)}(\ti{x},\ti{y})\tau_1^{3}|d\ti{x} \leq C 2^{-(4j-3k)}2^{-m} \mathcal{Q} \mathcal{P}^2
\int_{ \Gamma_{\ti{x}}} \chi_{j-k}(\ti{x})d\ti{x}\\
&{\hskip 5.1cm}=C 2^{-j}2^{-m} \mathcal{Q} \mathcal{P}^2\int_{ \Gamma_{x}} \chi_{0}(x)dx\\
&{\hskip 5.1cm}=C 2^{-(j-k)}2^{-k}2^{-m} \mathcal{Q}^2\mathcal{P}^2.
\end{split}
\end{equation}
Let $H(z):\mathbb{Z}\rightarrow \{0,1\}$ be defined by $H(z)=1$ if $|z|\leq 1$ and $H(z)=0$ if $|z|>1$. 
For the general $(j,l),(k,m)$ with $|j-k|\geq 3$, the $2^{-(j-k)}2^{-k}2^{-m}$
in the right hand side of~\eqref{E:mathcal-K-bound} should 
be replaced by 
\[
2^{-|j-k|} H(l-m) 2^{-\min\{j,k\}}2^{-\min\{l,m\}}.
\] 
We note that $2^{-\min\{j,k\}}2^{-\min\{l,m\}}\leq 1$. 
By invoking Schur test lemma (lemma~\ref{lemma:Schur_test}), there exists a constant $C$ such that
\begin{equation}\label{E:x-distance}
\|F_{(j,l)}F^{*}_{(k,m)}\|_{L^2\rightarrow L^2}\leq C 2^{-|j-k|} H(l-m) \mathcal{Q}^2 \mathcal{P}^2.
\end{equation}
Applying the same argument to the case $|l-m|\geq 3$, we have 
\begin{equation}\label{E:xi-distance}
\|F^{*}_{(j,l)}F_{(k,m)}\|_{L^2\rightarrow L^2}\leq C 2^{-|l-m|} H(j-k) \mathcal{Q}^2 \mathcal{P}^2.
\end{equation}
Next we prove that~\eqref{E:x-distance} and~\eqref{E:xi-distance} also hold for $|j-k|<3$ 
and $|l-m|<3$ respectively. By symmetry it suffices to prove one of them. Thus we assume  
$k\leq j<k+3$ and $l=m+1$ and it remains to prove  
\begin{equation}\label{E:x-distance-unit}
\|F_{(j,l)}F^{*}_{(k,m)}\|_{L^2\rightarrow L^2}\leq C  \mathcal{Q}^2 \mathcal{P}^2.
\end{equation}
When $|\ti{x}-\ti{y}|\geq (2\tau_1\tau_2)^{-1}$, similar to~\eqref{E:decay_kernel}, 
we can derive
\begin{equation}\label{E:decay_kernel-2}
\begin{split}
 |\mathcal{K}_{(j,l)(k,m)}(\ti{x},\ti{y})|\leq & C(\tau_1\tau_2)^{-4} \mathcal{P}^2\frac{\chi_{j-k}(\ti{x})\chi_0(\ti{y})}{|\ti{x}-\ti{y}|^4}\\
&{\hskip 2cm} \times\int_{\mathbb{R}^3}\chi_{0}|_{y\times\Gamma_{\xi}}(\ti{\xi})\chi_{1}|_{x\times\Gamma_{\xi}}(\ti{\xi})d\ti{\xi}. 
\end{split}
\end{equation}
Applying~\eqref{E:decay_kernel-2} to the the left hand side of~\eqref{E:mathcal-K-bound} and estimating the integrals by 
considering the region $|\ti{x}-\ti{y}|\approx (2^n\tau_1\tau_2)^{-1}$ for each $n\in\mathbb{N}$, we 
conclude that they have the upper bound $C \mathcal{Q}^2 \mathcal{P}^2$. When 
$|\ti{x}-\ti{y}|<  (2\tau_1\tau_2)^{-1}$, direct estimate gives
\begin{equation}\label{E:decay_kernel-3}
|\mathcal{K}_{(j,l)(k,m)}(\ti{x},\ti{y})|\leq C \mathcal{P}^2 {\chi_{j-k}(\ti{x})\chi_0(\ti{y})}
\cdot\int_{\mathbb{R}^3}\chi_{0}|_{y\times\Gamma_{\xi}}(\ti{\xi})\chi_{1}|_{x\times\Gamma_{\xi}}(\ti{\xi})d\ti{\xi}. 
\end{equation} 
Applying~\eqref{E:decay_kernel-3} to the the left hand side of~\eqref{E:mathcal-K-bound}
and considering the support of $\ti{x}$ and $\ti{y}$, we conclude~\eqref{E:x-distance-unit}.

Now we have 
\[
\|F_{(j,l)}F^{*}_{(k,m)}\|_{L^2\rightarrow L^2},\|F^{*}_{(j,l)}F_{(k,m)}\|_{L^2\rightarrow L^2} \leq C
\mathcal{Q}^2 \mathcal{P}^2 \{\Theta(j-k,l-m) \}^2,
\] 
where 
\[
\Theta(j_1,j_2)=\sqrt{\frac{H(j_2)}{2^{j_1}}+\frac{H(j_1)}{2^{j_2}}}
\]
Since 
\[
\sum_{(j_1,j_2)\in\mathbb{N}\cup\{0\}\times \mathbb{N}\cup\{0\}} \Theta(j_1,j_2) <\infty,
\]
we can conclude the result by Cotlar-Stein lemma (Lemma~\ref{lemma:Coltar's_lemma}).

\end{proof}

\section{Proof of Lemma~\ref{L:T-main-1}}

First we reduce the estimate~\eqref{E:T-main-1} to a $L^2$ to $L^2$ estimate which is more
natural and simply.  Recalling~\eqref{D:T-Fourier-1},  we define
\begin{equation}\label{D:a-x-xi}
 a(x,\xi)\eqdef \mathbb{A}(x,\xi)|\xi|^{\gamma}=(|x||\xi|)^{^{\gamma}}
 \int_{\omega\in S^2_+} e^{-i(x\cdot\omega)(\xi\cdot\omega)} b(\cos\theta) d\Omega(\omega)
\end{equation}
and 
\begin{equation}\label{D:T-Fourier}
Th(x)=\int_{\mathbb{R}^3} e^{i{x\cdot\xi}} a(x,\xi) \widehat{h}(\xi) d\xi.
\end{equation}
Recall 
\[
D^{-\gamma} f(x)=(2\pi)^{-3}\int_{\mathbb{R}^3} e^{i{x\cdot\xi}} |\xi|^{-\gamma}\widehat{f}(\xi)d\xi.
\]
From $\mathbb{T} h(x)=(2\pi)^{-3} T(D^{-\gamma} h)(x)$, we see that
the Lemma~\ref{L:T-main-1} is equivalent to 
\begin{lem}\label{L:T-main}
The operator $T$ defined in~~\eqref{D:T-Fourier} satisfies
\begin{equation}\label{E:T-main}
\| T h\|_{L^2}\leq C \|h\|_{L^2}.
\end{equation}
\end{lem}
\begin{proof}

For the analysis of $a(x,\xi)$ we need a dyadic partition of unity 
on the phase space $\{(x,\xi)|x,\xi\in\mathbb{R}^3-\{0\} \}$.
Using~\eqref{D:chi-k}, we have the dyadic partition of unity by
\[
\begin{split}
1&=\sum_{j\in \mathbb{Z}} \chi_{j}(x) \sum_{l\in \mathbb{Z}} \chi_{l}(\xi) ,\; x\neq0,\; \xi\neq 0\\
&=\chi_{A}(x,\xi)+\chi_{B}(x,\xi)+\chi_{C}(x,\xi)
\end{split}
\]
where
\begin{equation}\label{D:dyadic-partition}
\begin{split}
&\chi_{A}(x,\xi)=\sum_{j\in \mathbb{N}} \chi_{j}(x) \sum_{l\in \mathbb{N}} \chi_{l}(\xi), \\
&\chi_{B}(x,\xi)=\chi_{B,1}(x,\xi)+\chi_{B,2}(x,\xi)\\
&=\sum_{j+l\geq 2,j\geq 1, l \leq 0} \chi_{j}(x) \chi_{l}(\xi)+\sum_{j+l\geq 2,j\leq 0} \chi_{j}(x) \chi_{l}(\xi), \\
&\chi_{C}(x,\xi)=\chi_{C,1}(x,\xi)+\chi_{C,2}(x,\xi)\\
&=\sum_{j+l\leq 1,j\geq 1} \chi_{j}(x) \chi_{l}(\xi)+\sum_{j+l\leq 1,j\leq 0} \chi_{j}(x) \chi_{l}(\xi).
\end{split}
\end{equation}
Hence ${\rm supp}\;\chi_{A}(x,\xi)\subset \{|x|>8,|\xi|>8\}$, ${\rm supp}\;\chi_{B,1}(x,\xi)\subset \{|x||\xi|>64, |x|>8, |\xi|<16\}$
, ${\rm supp}\;\chi_{B,2}(x,\xi)\subset \{|x||\xi|>64, |x|<16\}$, ${\rm supp}\;\chi_{C,1}(x,\xi)\subset \{|x||\xi|<512, |x|>8 \}$
and ${\rm supp}\;\chi_{C,2}(x,\xi)\subset \{|x||\xi|<512, |x|<16 \}$.
Using~\eqref{D:dyadic-partition}, we decompose $a(x,\xi)$  as
\begin{equation}\label{D:a-decom-three}
\begin{split}
a &=a_{A}+a_{B,1}+a_{B,2}+a_{C,1}+a_{C,2}\\
&\eqdef\;\chi_{A}\cdot a+\chi_{B,1}\cdot a+\chi_{B,2}\cdot a+\chi_{C,1}\cdot a+\chi_{C,2}\cdot a.\\
\end{split}
\end{equation}
Plugging this decomposition into~\eqref{D:T-Fourier}, we write
$T=(T_{A}+T_{B,1}+T_{B,2}+T_{C,1}+T_{C,2})$ accordingly.
It is reduced to proving that all these operators are $L^2$ bounded.

\noindent{\bf Part I. Estimate of $T_{A}$}\par

For the analysis of $a_A(x,\xi)$, we need a dyadic decomposition in the interval $(0,\pi)$ which is constructed below.
Let $\zeta(\theta)\in C^{\infty}$ be supported in the interval  $(\pi/8,\pi/2)$ and satisfy $\sum_{z\in\mathbb{Z}} \zeta(2^{-z}\theta)=1$
for all $\theta>0$. Let $\zeta_n(\theta)=\zeta(2^n \theta)$, if $n\in\mathbb{N}$  and $\tilde{\zeta}(\theta)=1-\sum_{n\in\mathbb{N}}
\zeta_n(\theta)$.  Define $\zeta_0(\theta)$ equals $\tilde{\zeta}(\theta)$, if $0< \theta\leq \pi/2$ and $\zeta_0(\theta)=\zeta_0(\pi-\theta)$, 
if $\pi/2<\theta<\pi$. This extension of  $\zeta_0$ from $0<\theta\leq \pi/2$
to $\pi/2<\theta<\pi$ by reflection keeps $\zeta_0$ a smooth function since $\tilde{\zeta}$ equals $1$ near $\pi/2$. 
We also define $\zeta_{-n}(\theta)=\zeta_n(\pi-\theta)$ for $n\in\mathbb{N}$. Then we have the dyadic decomposition 
$1=\zeta_0(\theta)+\sum_{n\in\mathbb{N}}(\zeta_n(\theta)+\zeta_{-n}(\theta))$ in the interval $\theta\in (0,\pi)$. 
Abuse the notations, we define
\begin{equation}\label{D:chi-j}
\zeta_{0}(x,\xi)=\zeta_{0}(\arccos(\frac{x\cdot\xi}{|x||\xi|}))\;,\;\zeta_{\pm n}(x,\xi)=\zeta_{\pm n}(\arccos(\frac{x\cdot\xi}{|x||\xi|}))
\;,\;n\in\mathbb{N}.
\end{equation}
Thus the supports of $\zeta_0(x,\xi),\zeta_{n}(x,\xi),\zeta_{-n}(x,\xi)$ lie respectively
in the cones 
\begin{equation}\label{D:Gamma-n}
\begin{split}
 & \Gamma_0={\Big \{} (x,\xi)|\;  \frac{\pi}{8} < \arccos(\frac{x\cdot\xi}{|x||\xi|}) 
 < {\pi}-\frac{\pi}{8} {\Big \}}, \\
 & \Gamma_{n}= {\Big \{} (x,\xi)|\;  \frac{\pi}{2^{n+3}} < \arccos(\frac{x\cdot\xi}{|x||\xi|}) < \frac{\pi}{2^{n+1}} {\Big \}}\\
 & \Gamma_{-n}={\Big \{} (x,\xi)|\; \pi(1-\frac{1}{2^{n+1}}) < \arccos(\frac{x\cdot\xi}{|x||\xi|})< \pi(1-\frac{1}{2^{n+3}}) {\Big \}}.
\end{split}
\end{equation}
Define
\begin{equation}\label{D:a-j}
a_{z}(x,\xi)=\zeta_{z}(x,\xi)a_A(x,\xi),\;z\in\mathbb{Z}
\end{equation}
and write $a_A(x,\xi)=\sum_{z\in\mathbb{Z}}a_{z}(x,\xi)$.
Then we have $T_A=\sum_{z\in\mathbb{Z}} T_z$ where
\begin{equation}\label{D:T-j}
T_{z} h(x)=\int_{\mathbb{R}^3} e^{i{x\cdot\xi}} a_{z}(x,\xi) \widehat{h}(\xi) d\xi,\;z\in\mathbb{Z}.
\end{equation}

We furthermore decompose each $a_{z}(x,\xi)=\zeta_{z}(x,\xi)\chi_A(x,\xi)a(x,\xi),\;z\neq 0$ into three parts. 
More precisely, for each $z\neq 0$, we decompose $\chi_{A}(x,\xi)$ into three parts by selecting suitable disjoint subgroups 
of $(j,l)\in\mathbb{N}\times\mathbb{N}$ in~\eqref{D:dyadic-partition} so that each has support lies respectively in  
\begin{equation}\label{D:Gamma-three-regions}
\begin{split}
& \Gamma_{z, I}=\{(x,\xi)\in\Gamma_{z}, |x|>8\cdot 2^{|z|},|\xi|>8\cdot 2^{|z|}\}\\
& \Gamma_{z, II}=\{(x,\xi)\in\Gamma_{z}, |x||\xi|>8^2\cdot 2^{2|z|}, 8<|\xi|<4\cdot 8\cdot 2^{|z|}\}\\
&{\hskip 1cm}\cup\{(x,\xi)\in\Gamma_{z}, |x||\xi|>8^2\cdot 2^{2|z|}, 8<|x|<4\cdot 8\cdot 2^{|z|}\}\\ 
& \Gamma_{z, III}=\{(x,\xi)\in\Gamma_{z}, 64<|x||\xi|<16\cdot8^2\cdot 2^{2|z|}\}.
\end{split}
\end{equation}
Those partition functions are denoted by $\chi_{z,I},\chi_{z,II}$ and $\chi_{z,III}$ respectively.
Then we have the decomposition of $a_z(x,\xi),z\neq 0 $ as 
\begin{equation}\label{D:a-j-three}
\begin{split}
 a_{z}(x,\xi)&=\chi_{z,I}(x,\xi) a_{z}(x,\xi) +\chi_{z,II}(x,\xi) a_{z}(x,\xi)+\chi_{z,III}(x,\xi) a_{z}(x,\xi)\\
 &\eqdef a_{z,I}(x,\xi)+ a_{z,II}(x,\xi)+ a_{z,III}(x,\xi)
\end{split}
\end{equation}
and decomposition of $T_{A}$ as
\[
 T_{A}=T_0+\sum_{z\neq 0} (T_{z,I}+T_{z,II}+T_{z,III} ).
\]

We will see that $T_0$ is the sum of two Fourier integral operators and a smooth operator. 
Each operator $T_{z,I}$ is similar to $T_0$, each $T_{z,II}$ behaves like $T_{z,I}$ after change of variables 
and  $T_{z,III}$ degenerates.
 We define $\Gamma_{0,A}=\{(x,\xi)\in\Gamma_0,|x|>8,|\xi|>8\}$ and
\begin{equation}
\begin{split}
&{\rm region\;I}=(\bigcup_{z\neq 0} \Gamma_{z, I})\bigcup \Gamma_{0,A} , \\
&{\rm region\;II}=\bigcup_{z\neq 0} \Gamma_{z, II}\;,
\;{\rm region\;III}=\bigcup_{z\neq 0} \Gamma_{z, III}.
\end{split}
\end{equation}
Let the angle spanned by $x$ and $\xi$ be $\theta_0$, we have 
$|x||\xi|\cos^2(\theta_0/2)\sin^2(\theta_0/2)>C_1>1$ on region I  and II and 
$|x||\xi|\cos^2(\theta_0/2)\sin^2(\theta_0/2)<C_2$ on region III for fixed constants $C_1,C_2$.

\noindent{\bf (i)}. Estimate for $T_0+\sum T_{z,I}$.\par 

First we give the precise description of $a_0(x,\xi)$ and functions in~\eqref{D:a-j-three}.
When $(x,\xi)$ lies in region I, we will calculate $a(x,\xi)$ given in~\eqref{D:a-x-xi} by using the 
stationary phase formula (Theorem~\ref{L:stationary} from~\cite{Hor83} is recorded here).  
The calculation here is similar to that in~\cite{Lio94,JC12}. We claim that on region I we have 
\begin{equation}\label{E:a-I}
a(x,\xi)=c_1 e^{-i|x||\xi|\sigma_+(x,\xi)} p_{+}(x,\xi)
+c_2 e^{-i|x||\xi|\sigma_-(x,\xi)} p_{-}(x,\xi)+ s(x,\xi)
\end{equation}
where 
\[
\sigma_{\pm}(x,\xi)=\frac{1}{2}(\frac{x\cdot\xi}{|x||\xi|}\pm 1),
\]
$s \in SG^{-\infty,-\infty}$ is a SG symbol of the smooth operator and $p_{\pm}\in SG^{\gamma-1,\gamma-1}
\subseteq SG^{0,0}$ are SG symbols of order $(0,0)$ 
satisfying~\eqref{E:symbol-derivative-bound}. Recall the SG symbols are defined in the notation subsection and 
$p(x,\xi)\in SG^{0,0}$ means
\[
|\partial^{\alpha}_{x}\partial^{\beta}_{\xi}p(x,\xi)|\leq C_{\alpha\beta}\lr{x}^{0-|\alpha|}\lr{\xi}^{0-|\beta|}.
\]
We remark that in this paper the term "symbol of order m" always means "SG symbols of order (m,m)".  
For our estimate, it is no harm to drop the symbol $s(x,\xi)$ in~\eqref{E:a-I}.
Together with the construction of $a_0,a_{z,I},z\neq 0$ in~\eqref{D:a-j} and~\eqref{D:a-j-three} 
we may write,  
\begin{equation}\label{E:a-A-I}
\begin{split}
&a_0(x,\xi)=c_1 e^{-i|x||\xi|\sigma_+(x,\xi)} p_{0+}(x,\xi)
+c_2 e^{-i|x||\xi|\sigma_-(x,\xi)} p_{0-}(x,\xi)\\
&a_{z,I}(x,\xi)=c_1 e^{-i|x||\xi|\sigma_+(x,\xi)} p_{z+}(x,\xi)
+c_2 e^{-i|x||\xi|\sigma_-(x,\xi)} p_{z-}(x,\xi)
\end{split}
\end{equation}
where 
\[
p_{0\pm}(x,\xi)=\zeta_0\cdot \chi_A\cdot p_{\pm}(x,\xi)\;,\;p_{z,\pm}(x,\xi)=\zeta_z\cdot\chi_{z,I}\cdot p_{\pm}(x,\xi).
\]
By the definitions of $\zeta_z,\chi_A$ and $\chi_{z,I}$, we see that $p_{0\pm}(x,\xi),p_{z,\pm}(x,\xi)\in SG^{0,0}$ 
satisfy~\eqref{E:symbol-derivative-bound} where $C_{\alpha\beta}$ are independent of $z$.

Now we prove that ~\eqref{E:a-I} holds on region I. By  parametrization 
\begin{equation}\label{E:parametrization}
\left\{
\begin{array}{l}
 x=|x|(0,0,1) \\
 \xi=|\xi|(\cos\varphi_0\sin\theta_0,\sin\varphi_0\sin\theta_0,\cos\theta_0) \\
 \omega=(\cos\varphi\sin\theta,\sin\varphi\sin\theta,\cos\theta)\;,\;\theta\in[0,\pi/2],\varphi\in[0,2\pi].
\end{array}
\right.
\end{equation}
the phase function in definition of $a(x,\xi)$ can be written as
\[
(x\cdot\omega)(\xi\cdot\omega) 
=|x||\xi|(\cos\theta(\cos(\varphi-\varphi_0)\sin\theta_0\sin\theta+\cos\theta_0\cos\theta).
\]
We regard $|x||\xi|$ as the parameter and define
\begin{equation}\label{E:phase-theta}
\sigma(x,\xi;\omega)
=\cos\theta(\cos(\varphi-\varphi_0)\sin\theta_0\sin\theta+\cos\theta_0\cos\theta)
\end{equation}
be the phase function of integral~\eqref{D:a-x-xi} which is obvious smooth over $S^2$.
In order to use the stationary phase formula, we need to find the critical point of the phase function 
$\sigma(x,\xi;\omega)$ and its Hessian over the critical point.
It is known that we can calculate these quantities on the local coordinate  or equivalently on manifold 
by using covariant derivatives (see, for example,~\cite{BW97}). We will use the later approach here.
The parametrization~\eqref{E:parametrization} defines a mapping from $(\theta,\varphi)$ to sphere. 
Then $\{e_{\theta}=\frac{\partial}{\partial \theta},
e_{\varphi}=\frac{1}{\sin \theta}\frac{\partial}{\partial \varphi}\} $ 
is a basis on tangent bundle of sphere.
Recall that the gradient and Hessian are given by the formulas,
\begin{equation}\label{E:gradient-formula}
\bigtriangledown_{S^2}:=(\frac{\partial}{\partial\theta},\frac{1}{\sin\theta}\frac{\partial}{\partial \varphi})
\end{equation} and
\begin{equation}\label{E:Hessian-formula}
H_{S^2}:=\left[
\begin{array}{cc}
\frac{\partial^2}{\partial \theta^2} &\frac{1}{\sin\theta}\frac{\partial^2}{\partial\theta\partial\varphi}
-\frac{\cos\theta}{\sin^2\theta}\frac{\partial}{\partial\varphi} \\
\frac{1}{\sin\theta}\frac{\partial^2}{\partial\theta\partial\varphi}
-\frac{\cos\theta}{\sin^2\theta}\frac{\partial}{\partial\varphi}
 & \frac{1}{\sin^2\theta}\frac{\partial^2}
{\partial\varphi^2}+\frac{\cos\theta}{\sin\theta}\frac{\partial}{\partial\theta} 
\end{array}\right].\\
\end{equation}

Applying~\eqref{E:gradient-formula} to $|x||\xi|\sigma(x,\xi;\omega)$, we see that the critical points 
of this function  satisfy the formula
\begin{equation}\label{E:critical-point-relation}
  (x\cdot\omega)\xi+(\xi\cdot\omega)x=2(x\cdot\omega)(\xi\cdot\omega)\omega\;,\;\omega\in S^2.
\end{equation}
Note $ b(\cos\theta)\sin\theta$ vanishes when $x\cdot\omega=0$, 
thus we only have to look for critical points $\omega$ such that $x\cdot\omega\neq 0$. 
By relation~\eqref{E:critical-point-relation}, we have $\xi\cdot\omega\neq 0$. Hence $\omega$
belongs to the plane generated by $x$ and $\xi$. A simple calculation yields four critical points
$\omega_{+},\omega_{-},-\omega_{+},-\omega_{-}$ where 
\begin{equation}\label{E:critical-points}
\left\{
\begin{array}{l}
\omega_{+}=(\xi|x|+x|\xi|)|\xi|x|+x|\xi||^{-1}\\
\omega_{-}=(\xi|x|-x|\xi|)|\xi|x|-x|\xi||^{-1}
\end{array}
\right.
\end{equation}    
if $x$ and $\xi$ are not colinear.  
We also express them as
\begin{equation}\label{E:critical-points-theta}
\begin{split}
\omega_+=&(\cos\varphi_0\sin(\theta_0/2),\sin\varphi_0\sin(\theta_0/2),
\cos(\theta_0/2)), \\
\omega_-=&(\cos(\pi+\varphi_0)\sin((\pi-\theta_0)/2),\\
& \qquad \sin(\pi+\varphi_0)\sin((\pi-\theta_0)/2),\cos((\pi-\theta_0)/2)). 
\end{split}
\end{equation} 
If $x$ and $\xi$ are colinear, the four critical points are $\pm\frac{x}{|x|}=\pm\frac{\xi}{|\xi|}$
and $\pm\omega^{\perp}$ where $\omega^{\perp}$ is a unit vector orthogonal to $x$ (and $\xi$). We observe that if 
$\frac{|x\cdot\xi|}{|x||\xi|}=0$ or $1$, then for all these four critical points, $\frac{|x\cdot\omega|}{|x|}$ is $0$ or $1$. 
The case $x\cdot\omega=0$ has been excluded, while we may disregard the case 
$|x\cdot\omega|/|x|=1$ since $b(\cos\theta)\sin\theta=0$ at $\theta=0$ or $\pi$. 
Thus we only have to consider the critical points of the form~\eqref{E:critical-points}.  
Furthermore, the support of $b(\cos\theta)$ is in $[0,\pi/2]$, thus we only have to consider the contribution from 
two critical points, $\omega_+$ and $\omega_-$ when applying the stationary phase formula.

The Hessians of the phase function $\sigma(x,\xi;\omega)$ at $\omega_+$ and $\omega_-$ are 
\begin{equation}\label{E:Hessian}
\sigma''(x,\xi;\omega_{\pm})=\left[
\begin{array}{cc}
-2& 0\\
0& -2\sigma_{\pm}(x,\xi)
\end{array}
\right]
\end{equation}
where
\[
\sigma_{\pm}(x,\xi)=\frac{1}{2}(\frac{x\cdot\xi}{|x||\xi|}\pm 1)=\sigma(x,\xi;\omega_{\pm}).
\]
Since 
\[
|\det \sigma''(x,\xi;\omega_{\pm})|=2^2|\sigma_{\pm}|\;,\;{\sgn}\;\sigma''(\omega_{\pm})=\mp 2,
\]
the critical points do not degenerate.  

To calculate $a(x,\xi)$ we introduce a partition of unity 
on $S^2$ such that 
\begin{equation}\label{D:kappa-i}
\kappa_i\in C^{\infty}(S^2),\;
0\leq\kappa_i\leq 1,\;\sum_{i=1}^3 \kappa_i=1
\end{equation}
and $\kappa_i\equiv 1,\;i=1,2$ in a neighborhood of $\omega_{\pm}$.
By integration by parts, we see that 
\begin{equation}\label{E:a-g0-loc}
(|x||\xi|)^{\gamma}\int_{\omega\in S^2_+}b(\cos\theta)
e^{-i(x\cdot\omega)(\xi\cdot\omega)}\kappa_3(\omega)d\Omega(\omega)\in SG^{-\infty,-\infty}
\end{equation}
is the symbol of smooth operator, since the support of $b(\cos\theta)\kappa_3(\omega)$ is compact,
the phase function has no critical point on this support. This gives the last term of~\eqref{E:a-I}.
Note that if one extend $b(\cos\theta)$ evenly from $[0,\pi/2]$ to $[0,\pi]$, it again suffices to consider 
two critical points $\omega_{+},\omega_{-}$ since $(x\cdot\omega)(\xi\cdot\omega)$ and $b(\cos\theta)$ are even in $\omega$,  
the contributions to $a(x,\xi)$ from the critical points $-\omega_{+},-\omega_{-}$ are identical respectively 
to those from $\omega_{+},\omega_{-}$. 

We see from~\eqref{E:parametrization},~\eqref{E:critical-points} that if the angle spanned 
by $x$ and $\xi$ is ${\theta_0},0<{\theta_0}<\pi$,
then the angles spanned by $x$ and $\omega_{\pm}$ are ${\theta_0}/2$ and 
$(\pi-{\theta_0})/2$ respectively.
Applying stationary phase formula to~\eqref{E:a-g0-loc} with $\kappa_3$ replaced by $\kappa_i,1\leq i\leq 2$, we have
\begin{equation}\label{E:stationary}
\begin{split}
& a(x,\xi)-s(x,\xi)\\
&
= \pi (|x||\xi|)^{\gamma} |\det(|x||\xi|\sigma''(x,\xi,\omega_+))|^{-\frac{1}{2}}e^{\frac{\pi i}{4}{\rm\sgn\sigma''(\omega_+)}}\times\\
&{\hspace{2cm}} e^{-i|x||\xi|\sigma_+(x,\xi)}{\Big\{}  b(\cos(\frac{\theta_0}{2}))\sin(\frac{\theta_0}{2})+O(|x|^{-1}|\xi|^{-1}){\Big\}}\\
&+ \pi (|x||\xi|)^{\gamma} |\det(|x||\xi|\sigma''(x,\xi,\omega_-))|^{-\frac{1}{2}}e^{\frac{\pi i}{4}{\rm\sgn\sigma''(\omega_-)}}\times\\
&{\hspace{.5cm}}e^{-i|x||\xi|\sigma_-(x,\xi)}{\Big\{}  b(\cos(\frac{\pi}{2}-\frac{\theta_0}{2}))
\sin(\frac{\pi}{2}-\frac{\theta_0}{2}) +O(|x|^{-1}|\xi|^{-1}){\Big\}}\\
&= e^{-i|x||\xi|\sigma_+(x,\xi)}p_+(x,\xi)+e^{-i|x||\xi|\sigma_-(x,\xi)}p_{-}(x,\xi).
\end{split}
\end{equation}
Also we note that  the first term of $p_{\pm}(x,\xi)$  are respectively
\begin{equation}\label{E:s-plus}
\begin{split}
&-i\pi\;(|x||\xi|)^{\gamma} |\det(|x||\xi|\sigma''(x,\xi,\omega_+))|^{-\frac{1}{2}}   b(\cos(\theta_0/2))
\sin(\theta_0/2)\\
&=\frac{-i\pi\;(|x||\xi|)^{\gamma}\cos(\theta_0/2)\sin(\theta_0/2)}{2|x||\xi||1+x\cdot\xi/|x||\xi||^{1/2}} \\
&=\frac{-i\pi\;\cos(\theta_0/2)\sin(\theta_0/2)}{2|1+\cos\theta_0|^{1/2}}(|x||\xi|)^{\gamma-1}\\
&=-2^{-3/2}i\pi\;  \sin(\frac{\theta_0}{2})(|x||\xi|)^{\gamma-1},
\end{split}
\end{equation} 
\begin{equation}\label{E:s-minus}
\begin{split}
&i\pi\; (|x||\xi|)^{\gamma} |\det(|x||\xi|\sigma''(x,\xi,\omega_-))|^{-\frac{1}{2}}b(\cos({\pi}/{2}-{\theta_0}/{2}))\sin({\pi}/{2}-{\theta_0}/{2}) \\
&=\frac{i\pi\;(|x||\xi|)^{\gamma}\cos({\pi}/{2}-{\theta_0}/{2})\sin({\pi}/{2}-{\theta_0}/{2})}{2|x||\xi||1-x\cdot\xi/|x||\xi||^{1/2}}\\
&=\frac{i\pi\;\cos({\pi}/{2}-{\theta_0}/{2})\sin({\pi}/{2}-{\theta_0}/{2})}{2|1-\cos\theta_0|^{1/2}}(|x||\xi|)^{\gamma-1}\\
&= 2^{-3/2}i\pi\;\cos(\frac{\theta_0}{2})(|x||\xi|)^{\gamma-1}.\\
\end{split}
\end{equation} 

The terms~\eqref{E:s-plus} and~\eqref{E:s-minus} are symbols of order $0$ (more precisely $\gamma-1$) when $(x,\xi)\in {\rm supp}\chi_A$
as the $x$ or $\xi$ derivative of $\cos(\theta_0/2)$ or $\sin(\theta_0/2)$ descends one order in $x$ or $\xi$ 
respectively.  For example,
\[
\frac{\partial}{\partial \xi_j}\cos(\frac{\theta_0}{2})=-\frac{1}{2}\sin(\frac{\theta_0}{2})\frac{\partial}{\partial \xi_j}\theta_0
\]
where
\[
\begin{split}
\frac{\partial}{\partial \xi_j}{\theta_0}&=\frac{\partial}{\partial \xi_j} \arccos(\frac{x\cdot\xi}{|x||\xi|}) \\
&=\frac{-1}{\sqrt{1-(x\cdot\xi/|x||\xi|)^2}}\frac{\partial}{\partial \xi_j}
{\Big (}\frac{x\cdot \xi}{|x||\xi|} {\Big)}\\
&=\frac{-1}{\sin\theta_0}\frac{x_j-\cos\theta_0|x|\xi_i/|\xi|}{|x||\xi|}.
\end{split}
\]
Using~\eqref{E:parametrization}, we see that 
\begin{equation}\label{E:derivative-theta}
\frac{\partial}{\partial \xi_j}{\theta_0}=
\left\{
\begin{array}{l}
-\cos\theta_0\cos\phi_0 |\xi|^{-1}\;,j=1\\
-\cos\theta_0\sin\phi_0 |\xi|^{-1}\;,j=2\\
-\sin\theta_0 |\xi|^{-1}\;,j=3
\end{array}
\right.
\end{equation}
The big O of~\eqref{E:stationary} can be calculated explicitly using stationary phase 
asymptotics~\eqref{F:stationary-formula}.  
When $k\geq 2$, the $k$-th term of asymptotics have the form
\begin{equation}\label{E:lower-order}
\cos(\theta_0/2)\sin({\theta_0}/{2})\frac{q_{\pm k}(\cos(\theta_0/2),\sin({\theta_0}/{2}))}
{(|x||\xi|\cos^2(\theta_0/2)\sin^2({\theta_0}/{2}))^{k-1}}(|x||\xi|)^{\gamma-1}
\end{equation}
where $q_{\pm k}(t,s)$ are  polynomials of $(t,s)$. 
The remainder term of the stationary phase formula also has the form as above.  
On the region I, we have 
\[
|x||\xi|\cos^2(\theta_0/2)\sin^2(\theta_0/2)>C_1>1.
\] 
Using this fact, definition of region I together with $\sin(\theta_0/2),\cos(\theta_0/2)$ are symbols of order $0$, 
we know that ~\eqref{E:lower-order} are symbols of order $-k+1$ (precisely  $-k+1+\gamma-1$) on  region I. 
By the asymptotic sum of symbols,  we can conclude that $p_{\pm}(x,\xi)$ of~\eqref{E:a-A-I} are symbols 
of order $0$ on  region I. Furthermore $p_{\pm}(x,\xi)$ satisfy~\eqref{E:symbol-derivative-bound}  on region I.  
We remark that in order to estimate the $L^2$ bounds of the operators defined on region I or II 
induced by~\eqref{E:a-I}, one can use only the first four terms of the asymptotic expansion of the stationary phase 
formula instead of full series.  Since the operator induced by the remainder term has kernel of the form~\eqref{E:lower-order} with $k=5$    
which are integrable with respect to $x$ for fixed $\xi$ and vice verse. Thus the boundedness of the operator given by the remainder term
follows from Schur test. 

Using~\eqref{E:a-A-I}, we begin to estimate $T_0$ and $T_{z,I}$.  That is to consider 
the operators of the form~\eqref{D:T-j} with $a_z$ replaced by the right hand side of 
~\eqref{E:a-A-I}.  
Define
\[
\psi_{\pm}(x,\xi)=x\cdot\xi-|x||\xi|\sigma_{\pm}(x,\xi)=\frac{1}{2}(x\cdot\xi\mp |x||\xi|).
\]
By Plancherel theorem, we need to show that the operators 
\begin{equation}\label{D:model-FIO}
 F_{z\pm}h(x)=\int e^{i\psi_{\pm}(x,\xi)} p_{z\pm}(x,\xi){h}(\xi)d\xi,\; z\in\mathbb{Z}
\end{equation}
are $L^2$ bounded and their norms form a convergent series. 
We note that 
\begin{equation}\label{E:non-degeneracy-a}
|\det\partial_x\partial_{\xi} {\Big [} \psi_{\pm}(x,\xi)){\Big ]}|
=|\det\frac{1}{2}{\Big [}I\mp\frac{x}{|x|}\otimes\frac{\xi}{|\xi|}{\Big ]}|
= (\frac{1}{2})^3|1\mp\cos\theta_0|.
\end{equation} 
When $(x,\xi)\in{\rm supp\;}p_{0\pm}(x,\xi)\subset \Gamma_{0,A}$, there exists constants $C_1,C_2$ such that
\begin{equation}\label{E:non-degeneracy-a-0}
 0<C_1<(\frac{1}{2})^3|1\mp\cos\theta_0|<C_2.
\end{equation}
The phase functions $\psi_{\pm}(x,\xi)$ of $F_{0\pm}$ satisfy 
~\eqref{E:non_degenerate0} and~\eqref{E:phase-upper-bound}. Also
$p_{0\pm}$ satisfy~\eqref{E:symbol-derivative-bound}. Thus we conclude that 
$F_{0\pm}$ are $L^2$ bounded by Lemma~\ref{lemma:Fourier_integral}.

For the estimates of $F_{z\pm},\;z\neq 0$, we note that $1-\cos\theta_0$ tends to $0$ on the support 
of $p_{z+}(x,\xi)\subset \Gamma_{z,I}$ as $z$ tends to $\infty$, has uniform lower and upper 
bounds for $z<0$.
On the other hand,  $1+\cos\theta_0$ tends to $0$ on the support 
of $p_{z-}(x,\xi)$ as $z$ tends to $-\infty$, has uniform lower and upper bounds for $z>0$.
This symmetry and the form of the symbols $p_{\pm}(x,\xi)$ indicate that we only have to consider 
$F_{n\pm}$ with $n\in\mathbb{N}$. 
We note that the phase function 
$\psi_{-}(x,\xi)$ of $F_{n-}$ satisfies~\eqref{E:non_degenerate0},~\eqref{E:phase-upper-bound} of 
Lemma~\ref{lemma:Fourier_integral} and  has uniform lower and upper bounds on $\Gamma_{n-}$.
The amplitude functions $p_{n-}$ of $F_{n-}$ enjoy ~\eqref{E:symbol-derivative-bound}
with the same $C_{\alpha\beta}$.  Thus we may sum up $F_{n-}$ to obtain a new operator defined on region I,
then Lemma~\ref{lemma:Fourier_integral} implies that this operator is $L^2$ bounded.

For the estimate of the $L^2$ norms of $F_{n+}$, we need to introduce some notations. 
In~\eqref{E:L-2-bound} of  the Lemma~\ref{lemma:Fourier_integral}, the upper bound
of the $L^2$ norm of $F$ is denoted $\mathcal{C}\mathcal{Q}\mathcal{P}$. When we applying 
Lemma~\ref{lemma:Fourier_integral} to estimate the upper bounds of $L^2$ norms of $F_{n+}$
, we should use $\mathcal{C}_{n+}\mathcal{Q}_{n+}\mathcal{P}_{n+}$ to indicate its dependence on $n+$. 
The proof of the result that $F_{n+},\;n\in\mathbb{N}$ are $L^2$ bounded and their norms form a convergent series
is inferred by the following three results. 
The first one is that $F_{1+}$ is $L^2$ bounded, then the ratio $\mathcal{C}_{2+}\mathcal{Q}_{2+}\mathcal{P}_{2+}
/\mathcal{C}_{1+}\mathcal{Q}_{1+}\mathcal{P}_{1+}$ for $F_{2+}$ and $F_{1+}$ is less than 1 
and finally the calculation of above ratio works for any pair $F_{(n+1)+}$ and $F_{n+}$ with $n\in\mathbb{N}$.

The first result that $F_{1+}$ is $L^2$ bounded follows from Lemma~\ref{lemma:Fourier_integral} directly since
 $\psi_+$ satisfies conditions~\eqref{E:non_degenerate0},~\eqref{E:phase-upper-bound} on 
 ${\rm supp}\;p_{1+}\subset \Gamma_1$ and $p_{1+}$ satisfies~\eqref{E:symbol-derivative-bound}. 
Next we calculate the ratio $\mathcal{C}_{2+}\mathcal{Q}_{2+}\mathcal{P}_{2+}
/\mathcal{C}_{1+}\mathcal{Q}_{1+}\mathcal{P}_{1+}$. The constants $\mathcal{C}_{1+}$ and $\mathcal{C}_{2+}$
are determined by the constant $C$ in~\eqref{E:x-distance} which is from the constants $C$ 
of~\eqref{E:decay_kernel},~\eqref{E:decay_kernel-2} and~\eqref{E:decay_kernel-3}. 
Note that we do not need to consider the constant $C$ in~\eqref{E:xi-distance} by symmetry. 
The constants $C$ of~\eqref{E:decay_kernel} and~\eqref{E:decay_kernel-2} come from~\eqref{E:non_degenerate1} 
, i.e., the $C_1$ of~\eqref{E:non-degeneracy-a-0}. The $C_1$ on $\Gamma_{2,I}$ is about $1/4$
of that on $\Gamma_{1,I}$. On the other hand, the second term of~\eqref{E:non-degeneracy-a} and the 
parametrization~\eqref{E:parametrization} indicate that this $1/4$ decrease only occurs along the third 
component of $\xi$ (or $x$) when we use~\eqref{E:L-change} and integration by parts. The 
calculation~\eqref{E:derivative-theta} for $j=3$ suggests that actually we have $1/2$ decrease since 
$\sin\theta_0$ on the $\Gamma_{2,I}$ is about $1/2$ of that on $\Gamma_{1,I}$. By these observations, it 
is easy to check that the constants $C$ of~\eqref{E:decay_kernel} and~\eqref{E:decay_kernel-2} on $\Gamma_{2,I}$
is $16$ times of that on $\Gamma_{1,I}$ after applying integration by parts 4 times. 
From the  definition of $\Gamma_{1,I},\Gamma_{2,I}$ and the form of 
symbols~\eqref{E:s-plus},~\eqref{E:lower-order}, we see that 
\[
\mathcal{P}_{2+}\leq \frac{1}{2} \mathcal{P}_{1+}
\]
without counting the decrease of $\sin\theta_0$ from $\Gamma_{1,I}$ to $\Gamma_{2,I}$ mentioned before. Also by definition 
$\mathcal{Q}_{2+}$ is about 
$1/4$ of $\mathcal{Q}_{1+}$. Combining these facts together, we see that the ratio 
$\mathcal{C}_{2+}\mathcal{Q}_{2+}\mathcal{P}_{2+}
/\mathcal{C}_{1+}\mathcal{Q}_{1+}\mathcal{P}_{1+}$ is about $1/2$ in this case. The other case is that $C$
in~\eqref{E:x-distance} comes from $C$ of~\eqref{E:decay_kernel-3}, where the later is independent of $n+$. 
Thus the the ratio $\mathcal{C}_{2+}\mathcal{Q}_{2+}\mathcal{P}_{2+}/\mathcal{C}_{1+}\mathcal{Q}_{1+}\mathcal{P}_{1+}$
is about $1/8$ in this case. By the telescopic definition of $\Gamma_{n+}$, the above argument works for ratio  
$\mathcal{C}_{(n+1)+}\mathcal{Q}_{(n+1)+}\mathcal{P}_{(n+1)+}/\mathcal{C}_{n+}\mathcal{Q}_{n+}\mathcal{P}_{n+}$ 
as well for all $n$ and we conclude that $T_0+\sum_{z\neq 0} T_{z,I}$ is $L^2$ bounded.

\noindent{\bf (ii)} Estimate for $T_{z,II},z\neq 0$. \par
We should prove that the upper bound of  $L^2$ norm of $T_{z,II}$ is no more than $2$ times that of $T_{z,I}$, then the result
follows from the estimates of $T_{z,I}$ before.  The factor $2$ comes from the fact that the definition of 
$\Gamma_{z,II}$ contains $2$ pieces. We should prove that when $T_{z,II}$ is restricted to one of 
these piece, its upper bound of $L^2$ norm is bounded by that of $T_{z,I}$. Recall the definition 
\[
T_{z,II} h(x)=\int_{\mathbb{R}^3} e^{ix\cdot\xi} a_{z,II}(x,\xi)\widehat{h}(\xi)d\xi
\]
where $a_{z,II}(x,\xi)=\chi_{z,II}(x,\xi)a(x,\xi)$ and $\chi_{z,II}(x,\xi)$ has support on the region
\[
\begin{split}
\Gamma_{z,II}&=\{(x,\xi)\in\Gamma_{z}, |x||\xi|>8^2\cdot 2^{2|z|}, 8<|\xi|<4\cdot 8\cdot 2^{|z|}\}\\
& \cup\{(x,\xi)\in\Gamma_{z}, |x||\xi|>8^2\cdot 2^{2|z|}, 8< |x|<4\cdot 8\cdot 2^{|z|}\}.\\ 
\end{split}
\] 
We can write $\chi_{z,II}=\chi^1_{z,II}+\chi^2_{z,II}$ according to each piece of 
$\Gamma_{z,II}$ as
\[
\begin{split}
&\chi^1_{z,II}(x,\xi)\eqdef \zeta_z(x,\xi)\sum_{(j,l)\in I^1_{z,II}} \chi_j(x)\chi_l(\xi),\;\\
&\chi^2_{z,II}(x,\xi)\eqdef \zeta_z(x,\xi)\sum_{(j,l)\in I^2_{z,II}} \chi_j(x)\chi_l(\xi),
\end{split}
\]
where
\[
\begin{split}
& I^1_{z,II}=\{(j,l):j+l\geq 2|z|+2, 1\leq l\leq |z|\} \\
& I^2_{z,II}=\{(j,l):j+l\geq 2|z|+2, 1\leq j\leq |z|\}. 
\end{split}
\]
Following this, we have the decomposition 
\[
a_{z,II}(x,\xi)=a^1_{z,II}(x,\xi)+a^2_{z,II}(x,\xi)\eqdef \chi^1_{z,II}(x,\xi)a(x,\xi)+\chi^1_{z,II}(x,\xi)a(x,\xi)
\]
and $T_{z,II}=T^1_{z,II}+T^2_{z,II}$ accordingly. First we estimate $L^2$ norm of $T^1_{z,II}$.
By Plancherel theorem, the $L^2$ norm of $T^1_{z,II}$ is equivalent to that of the operator
\[
\mathcal{T}_{z,II} u(x) =\int_{\mathbb{R}^3} e^{ix\cdot\xi}\; a^1_{z,II}(x,\xi)\; {u}(\xi) d\xi. 
\] 
We should estimate it 
using change of variables and Lemma~\ref{lemma:Fourier_integral}. Write
\[
\mathcal{T}_{z,II} u(x) =\sum_{(j,l)\in I^1_{z,II}}\mathcal{T}_{z(j,l)} u(x)
=\sum_{(j,l)\in I^1_{z,II}}\int_{\mathbb{R}^3} e^{ix\cdot\xi}\; a_{z(j,l)}(x,\xi)\; {u}(\xi) d\xi
\]
with 
\[
a_{z(j,l)}(x,\xi)=\zeta_{z}(x,\xi)\chi_j(x)\chi_l(\xi)a(x,\xi).
\]
From Lemma~\ref{lemma:Fourier_integral}, we know that $L^2$ norm of $\mathcal{T}_{z,II}$ is determined by 
$\|\mathcal{T}_{z(j,l)}\mathcal{T}^*_{z(k,m)}\|_{L^2\rightarrow L^2}$ and 
$\|\mathcal{T}^*_{z(j,l)}\mathcal{T}_{z(k,m)}\|_{L^2\rightarrow L^2}$ where $(j,l),(k,m)\in I^1_{z,II}$.
By symmetry, it suffices to illustrate how to estimate 
$\|\mathcal{T}_{z(j,l)}\mathcal{T}^*_{z(k,m)}\|_{L^2\rightarrow L^2}$ by change of variables.

We begin with a useful observation. 
Since $1\leq l\leq |z|$, there exists $N\in\mathbb{N},\;N\leq |z|$ so that $l+N= |z|+1$. Let 
$x=2^N\ti{x},\xi=2^{-N}\ti{\xi}$ and note $x\cdot\omega/|x|=\ti{x}\cdot\omega/|\ti{x}|$ for $\omega\in S^2_+$.  
Thus we have 
\[
\begin{split}
a(x,\xi)&=(|x||\xi|)^{\gamma}\int_{\omega\in S^2_+} e^{-i(x\cdot\omega)(\xi\cdot\omega)} b(\cos\theta) d\Omega(\omega)\\
&=(|\ti{x}||\ti{\xi}|)^{\gamma}\int_{\omega\in S^2_+} e^{-i(\ti{x}\cdot\omega)(\ti{\xi}\cdot\omega)} b(\cos\theta) d\Omega(\omega)\\
&=a(\ti{x},\ti{\xi}).
\end{split}
\] 
From $(x\cdot\xi)/|x||\xi|=(\ti{x}\cdot\ti{\xi})/|\ti{x}||\ti{\xi}|$, we have
\[
\begin{split}
a_{z(j,l)}(x,\xi)&=\zeta_{z}(x,\xi)\chi_j(x)\chi_l(\xi)a(x,\xi)\\
&=\zeta_{z}(\ti{x},\ti{\xi})\chi_{j-N}(\ti{x})\chi_{l+N}(\ti{\xi})a(\ti{x},\ti{\xi})\\
&=a_{z(j-N,l+N)}(\ti{x},\ti{\xi}).
\end{split}
\]
Since
\begin{equation}\label{E:Gamma-II-relation}
j+l\geq 2|z|+2,
\end{equation}
 we have $j-N\geq |z|+1\;,\;l+N\geq |z|+1$ which means that new variables $(\ti{x},\ti{\xi})$ lie
in the region I. Thus we can calculate $a_{z(j-N,l+N)}(\ti{x},\ti{\xi})$ as {\bf(i)} and conclude that 
they are the form of the second line of~\eqref{E:a-A-I}. 

Note that 
\[
\mathcal{T}_{z(j,l)}\mathcal{T}^*_{z(k,m)} u(x)=\int K_{z(j,l),(k,m)}(x,y)u(y)dy,
\]
where 
\[
 K_{z(j,l),(k,m)}(x,y)
=\int e^{i(x\cdot\xi-y\cdot\xi)} a_{z(j,l)}(x,\xi)\overline{a_{z(k,m)}(y,\xi)}
d\xi.
\]
The kernel non-vanishes only $l=m-1,l=m$ or $l=m+1$, we may assume $l=m+1$. 
Then the observation in the previous paragraph indicates that we can apply change of variables
according to $a_{z(k,m)}$, so that 
\[
\begin{split}
&K_{z(j,l),(k,m)}(x,y)\\
&=\int e^{i(\ti{x}\cdot{\ti\xi}-\ti{y}\cdot\ti{\xi})} a_{z(j-N,l+N)}(\ti{x},\ti{\xi})\overline{a_{z(k-N,m+N)}(\ti{y},\ti{\xi})}
2^{-3N}d\ti{\xi}\\
&= 2^{-3N}K_{z(j-N,l+N),(k-N,m+N)}(\ti{x},\ti{y})
\end{split}
\]
where $j-N,l+N,k-N,m+N$ are all greater than $|z|+1$. This means the kernel $K_{z(j-N,l+N),(k-N,m+N)}(\ti{x},\ti{y})$ can be calculated 
explicitly as  $a_{z(j-N,l+N)},\;a_{z(k-N,m+N)}$ have explicitly formulas by the calculation in {\bf (i)}.
Note
\[
\begin{split}
\mathcal{T}_{z(j,l)}\mathcal{T}^*_{z(k,m)}u(2^N\ti{x})=\int K_{z(j-N,l+N),(k-N,m+N)}(\ti{x},\ti{y})u(2^N\ti{y})d\ti{y}.
\end{split}
\]
This means 
\[
\|\mathcal{T}_{z(j,l)}\mathcal{T}^*_{z(k,m)}\|_{L^2_y\rightarrow L^2_x}=
\|\mathcal{T}_{z(j-N,l+N)}\mathcal{T}^*_{z(k-N,m+N)}\|_{L^2_{\ti{y}}\rightarrow L^2_{\ti{x}}},
\]
and the later can be estimated as we did in Lemma~\ref{lemma:Fourier_integral} whose quantity is controlled 
by $z$, $|(j-N)-(k-N)|=|j-k|$ and $|(l+N)-(m+N)|=|l-m|$, independent of $N$.
Thus $\|{T}^1_{z,II}\|_{L^2\rightarrow L^2}=\|\mathcal{T}_{z,II}\|_{L^2\rightarrow L^2}\leq \|T_{z,I}\|_{L^2\rightarrow L^2}$. 
The proof of $\|{T}^2_{z,II}\|_{L^2\rightarrow L^2}\leq \|T_{z,I}\|_{L^2\rightarrow L^2}$ is essential the same,  we
skip it.

\noindent {\bf (iii)}.  Estimate for $T_{z,III},z\neq 0$. \par  

By Plancherel theorem, it is equivalent to write 
\[
\begin{split}
T_{z,III} h(x)&=\int_{\mathbb{R}^3} e^{ix\cdot\xi} a_{z,III}(x,\xi)\; h(\xi) d\xi \\
&= \int_{\mathbb{R}^3} K_{z,III}(x,\xi)\; h(\xi) d\xi
\end{split}
\]
where $a_{z,III}(x,\xi)=\chi_{z,III}(x,\xi) a_z(x,\xi)=\chi_{z,III}(x,\xi) \zeta_{z}(x,\xi)  a(x,\xi)$. 

Recall that  when $(x,\xi)\in \Gamma_{z,III}$ we have
\begin{equation}\label{E:III-relation}
|x||\xi|\cos^2(\theta_0/2)\sin^2(\theta_0/2)<C_2,
\end{equation}
and this means the stationary phase formula is not a good tool to calculate $a_{z,III}(x,\xi)$. 
On the other hand the calculation of $a_{z,I}(x,\xi)$ still gives some hint, thus we continue using 
the nations there. Recall 
\[
a(x,\xi)=(|x||\xi|)^{\gamma}\int_{\omega\in S^2_+} e^{-i(x\cdot\omega)(\xi\cdot\omega)} b(\cos\theta) d\Omega(\omega)
\]
and $\omega_{\pm}$ are the critical points of phase function $(x\cdot\omega)(\xi\cdot\omega)$.  
Use~\eqref{E:parametrization},~\eqref{E:critical-points-theta} and consider a smooth partition on $S^2$ as
\[
\overline{\kappa}_i\in C^{\infty}(S^2)\;,\; 0\leq\overline{\kappa}_i\leq 1\;,\;\sum_{i=1}^3\overline{\kappa}_i=1
\]
where $\overline{\kappa}_1\equiv 1$ when 
\[
(\theta,\varphi)\in (\varphi_0-{C_2}{(|x||\xi|)^{-1/2}},\varphi_0+{C_2}{(|x||\xi|)^{-1/2}})\times [\;0,{C_2}{(|x||\xi|)^{-1/2}})
\]
and $\overline{\kappa}_1\equiv 0$
when 
\[
(\theta,\varphi)\notin (\varphi_0-2{C_2}{(|x||\xi|)^{-1/2}},\varphi_0+2{C_2}{(|x||\xi|)^{-1/2}})\times [\;0,2{C_2}{(|x||\xi|)^{-1/2}}).
\]
And $\overline{\kappa}_2\equiv 1$ when 
\[
(\theta,\varphi)\in (\varphi_0-{C_2}{(|x||\xi|)^{-1/2}},\varphi_0+{C_2}{(|x||\xi|)^{-1/2}})\times 
(\pi/2-{C_2}{(|x||\xi|)^{-1/2}},\pi/2]
\]
and $\overline{\kappa}_2\equiv 0$
when 
\[
(\theta,\varphi)\notin (\varphi_0-2{C_2}{(|x||\xi|)^{-1/2}},\varphi_0+2{C_2}{(|x||\xi|)^{-1/2}})\times 
(\pi/2-2{C_2}{(|x||\xi|)^{-1/2}},\pi/2].
\]
From the relation~\eqref{E:III-relation} and constructions of $\overline{\kappa}_1, \overline{\kappa}_2$, we know that 
$\omega_{\pm}$ must lie in the supports of $\overline{\kappa}_1, \overline{\kappa}_2$ respectively and the distances 
of $\omega_{\pm}$ to the boundary of support are of order $(|x||\xi|)^{-1/2}$. Hence there exists $C>0$ such that for 
$\omega\in {\rm supp}\overline{\kappa}_3$ we have
\[
|\nabla_{S^2} [(x\cdot\omega)(\xi\cdot\omega)]|\geq C(|x||\xi|)^{1/2}
\]
where $\nabla_{S^2}$ is the gradient on the sphere given by~\eqref{E:gradient-formula}. Therefore the contribution from 
the support of $\overline{\kappa}_3$ to $a(x,\xi)$ is bounded by $(|x||\xi|)^{-n}$ for any $n\in\mathbb{N}$. On the 
supports of $\overline{\kappa}_1,\overline{\kappa}_2$ we have $\cos\theta\sin\theta\leq C(|x||\xi|)^{-1/2}$. Hence 
their contributions to $a(x,\xi)$ is bounded by $C(|x||\xi|)^{-1/2+(\gamma-1)}$. In summary, we have 
\begin{equation}\label{E:a-III-bound}
|a(x,\xi)|\leq C(|x||\xi|)^{-1/2},
\end{equation}
and thus
\begin{equation}\label{E:K-III}
|K_{z,III}(x,\xi)|\leq C\chi_{z,III}(x,\xi) \zeta_{z}(x,\xi) (|x||\xi|)^{-1/2}.
\end{equation}
Let $z=\pm 1$ temporarily.  Let $p(\xi)=|\xi|^{-3/2}$ and $q(x)=|x|^{-3/2}$. For any fixed $x_0$ ($|x_0|>8$), using spherically
coordinate, we have   
\begin{equation}\label{E:Schur-K-III}
\begin{split}
& \int |K_{\pm 1,III}(x_0,\xi)| p(\xi)  d\xi \\
& = C_1|x_0|^{-1/2} \int_8^{4096/|x_0|} r^{-1/2}\cdot r^{-3/2} \cdot r^2 dr \\
&\leq C_2 |x_0|^{-3/2}=C_2q(x_0). 
\end{split}
\end{equation}
Similarly, 
For any fixed $\xi_0$ ($|\xi_0|>8$) we have   
\[
\begin{split}
& \int |K_{\pm 1,III}(x,\xi_0)| q(\xi)  dx \\
& = C_1|\xi_0|^{-1/2} \int_8^{4096/|\xi_0|} r^{-1/2}\cdot r^{-3/2} \cdot r^2 dr \\
&\leq C_2 |\xi_0|^{-3/2}=C_2p(\xi_0). 
\end{split}
\]
where $C_1,C_2$ are the same as those in~\eqref{E:Schur-K-III} by symmetry. 
Thus $T_{\pm 1, III}$ are $L^2$ bounded by Schur test. To see that $T_{z,III},\;z\neq \pm 1$ are also $L^2$ bounded and their
norms form a convergent series we need to track $C_1,C_2$ of~\eqref{E:Schur-K-III} as $|z|$ varies. 
When $|z|$ goes from $n,n\in\mathbb{N}$ to  $n+1$, the $\zeta_z$ part of~\eqref{E:K-III} indicates that $C_1$ 
of the former is about $1/4$ of that for later. On the other hand, the $\chi_{z,III}$ part of~\eqref{E:K-III} or definition 
of $\Gamma_{z,III}$ indicates the upper limit in the integration of the second line of ~\eqref{E:Schur-K-III} 
increases $4$ times from $n$ to $n+1$. Therefore $C_2$ of the former is the same as that for the later.  To get 
a convergent series of $C_2$s, we may adjust the definitions of $\Gamma_{z,I}, \Gamma_{z,II},\Gamma_{z,III}$ 
in~\eqref{D:Gamma-three-regions} by replacing $2^{|z|},2^{2|z|}$ with $2^{(1-\delta)|z|},2^{2(1-\delta)|z|}$  for some 
small positive $\delta<1/14$. This adjustment will not affect the result of estimates in {\bf (i)} and {\bf (ii)}.  
This is because as we remarked in the end of the paragraph after~\eqref{E:lower-order}, we may drop the lower order terms of the 
stationary phase asymptotics and the fact that we  use integration by parts only 4 times in 
Lemma~\ref{lemma:Fourier_integral}.  Hence the ratio $\mathcal{P}_{(n+1)+}/\mathcal{P}_{n+}<2^{-(1-12\delta)}$
 after this adjustment.  Also we see that $C_2$ for $n+1$ is $2^{-2\delta}$ times of 
that for $n$ after the adjustment and  the result follows.

\noindent{\bf Part II. Estimate of $T_{B}$}\par

The proof of the result that $T_B$ is $L^2$ bounded is essential the reminiscence of Part I. 
Recall that 
\[
T_{B} h(x)=\sum_{j=1}^2 T_{B,j} h(x)=\sum_{j=1}^2\int_{\mathbb{R}^3} e^{ix\cdot\xi} \chi_{B,j}(x,\xi)a(x,\xi) 
\widehat{h}(\xi)d\xi 
\]
where  ${\rm supp}\;\chi_{B,1}(x,\xi)\subset \{|x||\xi|>64, |x|>8, |\xi|<16\}\;,\;
{\rm supp}\;\chi_{B,2}(x,\xi)\subset \{|x||\xi|>64, |x|<16\}$. 

Similarly to~\eqref{D:Gamma-three-regions}, we split operators $T_{B,1},T_{B,2}$ into sum of operators by 
decomposing the supports of $\chi_{B,1},\chi_{B,2}$ into cones according to~\eqref{D:Gamma-n}. 
Each cone is further split into three regions by letting $z\in\mathbb{Z}$ and defining
\begin{equation}\label{D:Gamma-three-regions-II}
\begin{split}
& \Gamma_{z, I}=\{(x,\xi)\in\Gamma_{z}, |x|>8\cdot 2^{|z|},|\xi|>8\cdot 2^{|z|}\}\\
& \Gamma_{z, II}=\{(x,\xi)\in\Gamma_{z}, |x||\xi|>8^2\cdot 2^{2|z|}, |\xi|<4\cdot 8\cdot 2^{|z|}\}\\
&{\hskip 1cm}\cup\{(x,\xi)\in\Gamma_{z}, |x||\xi|>8^2\cdot 2^{2|z|}, |x|<4\cdot 8\cdot 2^{|z|}\}\\ 
& \Gamma_{z, III}=\{(x,\xi)\in\Gamma_{z}, 64<|x||\xi|<16\cdot8^2\cdot 2^{2|z|}\}. 
\end{split}
\end{equation}
Then we define 
\begin{equation}
\begin{split}
&{\rm region\;I}_B=(\;\bigcup_{z} {\Gamma}_{z, I})\bigcap ({\rm supp}\;\chi_{B,1}\cup {\rm supp}\;\chi_{B,2}) \\ 
&{\rm region\;II}_B=(\;\bigcup_{z} \Gamma_{z, II})\bigcap ({\rm supp}\;\chi_{B,1}\cup {\rm supp}\;\chi_{B,2}) \\
&{\rm region\;III}_B=(\;\bigcup_{z} \Gamma_{z, III})\bigcap ({\rm supp}\;\chi_{B,1}\cup {\rm supp}\;\chi_{B,2}).
\end{split}
\end{equation}
From the supports of $\chi_{B,1},\chi_{B,1}$ we know that on the region $I_B$, only $\Gamma_{\pm 1,I}$ has 
non-empty intersection with the set $({\rm supp}\;\chi_{B,1}\cup {\rm supp}\;\chi_{B,2})$. Thus only two 
operators defined on region $I_B$ have non-zero values.  Also these two operators are $L^2$ bounded by 
the argument in {\bf (i)} of Part I.  For the operators defined on region $II_B$, we note that the definition of 
$\Gamma_{z,II}$ in~\eqref{D:Gamma-three-regions-II} is slightly different with that 
in~\eqref{D:Gamma-three-regions}. The conditions $|\xi|>8$ in the first set and $|\xi|>8$ in the second set 
of the $\Gamma_{z,II}$ in~\eqref{D:Gamma-three-regions} were removed now. However these two conditions in 
~\eqref{D:Gamma-three-regions} were used to emphasize the sets are part of support of $\chi_{A}$ but not 
used in the proof of {\bf (ii)} of Part I. 
Indeed, it is the condition $|x||\xi|>8^2\cdot 2^{2|z|}$, i.e.,~\eqref{E:Gamma-II-relation} ensures the argument 
in the proof of {\bf (ii)} works. Thus the operators defined on region $II_B$ are $L^2$ bounded and their norms 
form a convergent series as before. Also the operators defined on region $III_B$ can be treated exactly as 
{\bf (iii)} of Part I and we finish the proof of Part II.

\noindent{\bf Part III. Estimate of $T_{C}$}\par

By Plancherel Theorem, the $L^2$ boundedness of $T_C$ is equal to that of
\[
\mathcal{T}h(x)=\int_{\mathbb{R}^3} e^{ix\cdot\xi} \; a_{C}(x,\xi)\; h(\xi)d\xi.
\]
Recall that  
\[
\begin{split}
a_{C}(x,\xi)&=a_{C,1}(x,\xi)+a_{C,2}(x,\xi)\\
&=\chi_{C,1}(x,\xi)(|x||\xi|)^{\gamma}\int_{\omega\in S^2}e^{-i(x\cdot\omega)
(\xi\cdot\omega)} b(\cos\theta) d\Omega(\omega) \\
&{\hskip 1cm}+ \chi_{C,2}(x,\xi)(|x||\xi|)^{\gamma}\int_{\omega\in S^2}e^{-i(x\cdot\omega)
(\xi\cdot\omega)} b(\cos\theta) d\Omega(\omega)
\end{split}
\]
where ${\rm supp}\chi_{C,1}\subset \{|x||\xi|<512,|x|>8\}$ and ${\rm supp}\chi_{C,2}\subset \{|x||\xi|<512,|x|<16\}$.
Write the operator $\mathcal{T}$ as 
\[
\begin{split}
\mathcal{T}h(x)&=\mathcal{T}_1h(x)+\mathcal{T}_2h(x)\\
&=\int_{\mathbb{R}^3} e^{ix\cdot\xi} \; a_{C,1}(x,\xi) h(\xi)d\xi+\int_{\mathbb{R}^3} e^{ix\cdot\xi} \; a_{C,2}(x,\xi) h(\xi)d\xi\\
&= \int_{\mathbb{R}^3} K_1(x,\xi) h(\xi)d\xi +\int_{\mathbb{R}^3} K_2(x,\xi) h(\xi)d\xi
\end{split}
\]
It is clear from the support of $\chi_{C,j}$ that $K_j,\;i=1,2$ satisfies
\begin{equation}\label{E:kernel}
|K_{j}(x,\xi)|\leq C\cdot\chi_{C,j}(x,\xi).
\end{equation}
We should employ the Schur test  and 
consider two different sub-cases to prove 
that $\mathcal{T}$ is $L^2$ bounded. 

\noindent (a) boundedness of  $\mathcal{T}_1$. \par
 Let $p(x)=(1+|x|)^{-1}$ and $q(\xi)=|\xi|^{-2}$. For any fixed $|x_0|> 8$, 
using spherically coordinate and~\eqref{E:kernel}, we have 
\[
\begin{split}
\int_{\mathbb{R}^3} |K_1(x_0,\xi)|q(\xi)d\xi
&\leq C_1 \int_0^{512 |x_0|^{-1}} r^{-2}\cdot r^2 dr\\
&\leq C_2|x_0|^{-1}\leq C_3p(x_0).
\end{split}
\]
And for any fixed $|\xi_0|\leq 64$, we have 
\[
\begin{split}
\int_{\{\mathbb{R}^3,|x|>8\}} |K_1(x,\xi_0)|p(x)dx
&\leq C_1 \int_8^{512|\xi_0|^{-1}} r^{-1}\cdot r^2 dr\\
&\leq C_2|\xi_0|^{-2}= C_2q(\xi_0).
\end{split}
\]
By Schur test, we conclude the $L^2$ boundedness of this case.

\noindent (b)  boundedness of  $\mathcal{T}_2$.  \par

Let $p(x)=|x|^{-1}$ and $q(\xi)=(1+|\xi|)^{-2}$. For any fixed $|x_0|<16$,
using spherical coordinate and~\eqref{E:kernel} ,we have 
\[
\begin{split}
\int_{\mathbb{R}^3} |{K}_2(x_0,\xi)|q(\xi)d\xi
&\leq C_1 \int_0^{512|x_0|^{-1}} r^{-2}\cdot r^2dr\\
&\leq C_2|x_0|^{-1}=C_2p(x_0).
\end{split}
\]
For any fixed $|\xi_0|> 32$, we have 
\[
\begin{split}
\int_{\mathbb{R}^3} |{K}_2(x,\xi_0)|p(x)dx
&\leq C_1 \int_0^{512|\xi_0|^{-1}} r^{-1}\cdot r^{2} dr\\
&\leq {C_2}|\xi_0|^{-2}\leq C_3 q(\xi_0).
\end{split}
\]
And for any fixed $|\xi_0|\leq 32$, we have 
\[
\begin{split}
\int_{\{|x|\leq 16\}} |{K}_2(x,\xi_0)|p(x)dx
&\leq C_1 \int_0^{16} r^{-1}\cdot r^{2} dr\\
&\leq C_2\leq {C_3}q(\xi_0).
\end{split}
\]
By Schur test, we conclude that $\mathcal{T}$ is bounded. 

\end{proof}

\section{Proof of Lemma~\ref{L:T-v-v_*-ineq} and Lemma~\ref{L:H-v-v_*-ineq} and Corollary~\ref{C:Q-L}}\label{Proof of lemmas}

With the help of the proof of Lemma~\ref{L:T-main}, we can now prove Lemma~\ref{L:T-v-v_*-ineq} 
and Lemma~\ref{L:H-v-v_*-ineq} easily.

\begin{proof}[Proof of Lemma~\ref{L:T-v-v_*-ineq}]

Following~\eqref{D:T-Fourier-1},\eqref{D:a-x-xi} and~\eqref{D:T-Fourier}, we write 
\begin{equation}\label{E:T-translation}
\begin{split}
&(\tau_{-v_*}\circ \mathbb{T}\circ \tau_{v_*})\; h(v)\\
&=|v-v_*|^{\gamma}\int_{\omega\in S^2_+} b(\cos\theta)  h(v-
((v-v_*)\cdot\omega)\;\omega)d\Omega(\omega)\\
&=(2\pi)^{-3}\int_{\mathbb{R}^3}e^{iv\cdot\xi}\;\mathbb{A}(v-v_*,\xi)|\xi|^{\gamma}\;(|\xi|^{-\gamma}\widehat{h}(\xi))\;d\xi
\end{split}
\end{equation}
where 
\begin{equation}
\cos\theta=(v-v_*,\omega)/|v-v_*|,\;0\leq\theta\leq \pi/2,
\end{equation}
\begin{equation}\label{A-v-v-*}
\mathbb{A}(v-v_*,\xi)=|v-v_*|^{\gamma}\int_{S^2_+} b(\cos\theta)\;e^{-i((v-v_*)\cdot\omega)(\xi\cdot\omega)} d\Omega(\omega)
\end{equation}
Let 
\begin{equation}\label{D:a-v-v-*-xi}
a(v-v_*,\xi)=\mathbb{A}(v-v_*,\xi)|\xi|^{\gamma}.
\end{equation}
Then we are reduced to proving the $L^2$ boundedness of the operator 
\begin{equation}\label{D:T-v-v-*}
\int_{\mathbb{R}^3}e^{iv\cdot\xi} a(v-v_*,\xi)\; \widehat{h}(\xi)\;d\xi
\end{equation}
when it is regarded it as an operator of $v$ variable with $v_*$ as a parameter and vice versa.
Also we need to check the bounds are independent of the parameters.

The calculation of $a(x,\xi)$ in the proof of Lemma~\ref{L:T-main} can be applied to
that of $a(v-v_*,\xi)$ with $v-v_*$ playing the role of $x$. 
We may summarize that calculation of  $a$ in Lemma~\ref{L:T-main} as the following two cases. 
The first case is that $a$ has explicitly representation as~\eqref{E:a-I} (may need change of variables)
which induces the FIOs and their $L^2$ bounds are obtained by Lemma~\ref{lemma:Fourier_integral}. 
The second case is that $a(v-v_*,\xi)$ has explicitly upper bound as~\eqref{E:a-III-bound} or $a(v-v_*,\xi)$ is bounded as Part 
III of the proof of Lemma~\ref{L:T-main}. Then the 
operator induced by that can be estimated by Schur test directly. The second case of $a(v-v_*,\xi)$ for our current
estimates clearly give us the same result as the Lemma~\ref{L:T-main} no matter $v$ or $v_*$ is the 
variable of the operator. When the first case occurs, the operator~\eqref{D:T-v-v-*} is the sum of the 
operators of the following two forms, 
\begin{equation}\label{E:expression}
\int_{\mathbb{R}^3} e^{iv\cdot\xi}e^{-\frac{i}{2} [(v-v_*)\cdot\xi\pm|v-v_*||\xi|]}
\;p_{\pm}(v-v_*,\xi)\;\widehat{h}(\xi)d\xi 
\end{equation}
For any fixed $v_*$, we write~\eqref{E:expression}  as
\begin{equation}\label{E:expression-1}
\begin{split}
&\int_{\mathbb{R}^3} e^{\frac{i}{2}[(v-v_*)\cdot\xi\mp |v-v_*||\xi|]} p_{\pm}(v-v_*,\xi)
\; e^{iv_*\cdot\xi}\;  \widehat{h(\xi)}  d\xi\\
&=\int_{\mathbb{R}^3} e^{\frac{i}{2}[(v-v_*)\cdot\xi\mp |v-v_*||\xi|]} p_{\pm}(v-v_*,\xi)
\; \widehat{\tau_{-v_*}h}(\xi)  d\xi,
\end{split}
\end{equation}
and any fixed $v$, we write~\eqref{E:expression}  as
\begin{equation}\label{E:expression-2}
\int_{\mathbb{R}^3} e^{\frac{i}{2}[(v_*-v)\cdot\xi\mp |v_*-v||\xi|]} p_{\pm}(v-v_*,\xi)
\; \widehat{\tau_{-v}h}(\xi)  d\xi.
\end{equation}
From the facts that the translation is $L^2$ invariant, ~\eqref{E:expression-1} and~\eqref{E:expression-2}
have the same forms as~\eqref{D:model-FIO}, we conclude that the no matter $v$ or $v_*$ is the 
variable of the operator, the proof of Lemma~\ref{L:T-main} still works here for first case.    

\end{proof}

\begin{proof}[Proof of Lemma~\ref{L:H-v-v_*-ineq}]
First we prove the estimates~\eqref{E:T-c-v-ineq} and~\eqref{E:T-c-v_*-ineq}. Comparing~\eqref{Def:T}
and~\eqref{Def:T-c} and recalling the proof of Lemma~\ref{L:T-main-1} and Lemma~\ref{L:T-main}, 
we see that $\mathbb{T}_{\mathbbm{s}}$ also satisfies~\eqref{E:T-main-1}. By the proof of 
Lemma~\ref{L:T-v-v_*-ineq} above we know that $\mathbb{T}_{\mathbbm{s}}$ also 
satisfies~\eqref{E:T-v-ineq} and~\eqref{E:T-v_*-ineq}. Thus we only need to prove that when 
$\mathbb{T}_{\mathbbm{s}}$ is restricted to the low Fourier frequency ( $|\xi|<C$ in~\eqref{D:T-Fourier-1} ), 
denoted it by $\mathbb{T}_{\mathbbm{s}L}$, then the followings hold
\[
\begin{split}
&\sup\limits_{v_*}\|(\tau_{-v_*}\circ \mathbb{T}_{\mathbbm{s}L}\circ\tau_{v_*} )h(v)\|_{L^2(v)}\leq C \|h\|_{L^2},\\
&\sup\limits_{v}\|(\tau_{-v_*}\circ \mathbb{T}_{\mathbbm{s}L}\circ\tau_{v_*} )h(v)\|_{L^2(v_*)}\leq C \|h\|_{L^2}.
\end{split}
\] 
Since $(\tau_{-v_*}\circ \mathbb{T}_{\mathbbm{s}L}\circ\tau_{v_*} )h(v)$ has representation
~\eqref{E:T-translation} with $v-v_*$ and $\xi$ being restricted to compact sets respectively, 
these operators are clearly $L^2$ bounded.   

We turn to the proof of~\eqref{E:H-v-ineq} and~\eqref{E:H-v_*-ineq}. 
We denote
\[
|x|^{\gamma}_{\overline{\mathbbm{s}}}=(1-\mathbbm{s})(|x|)\cdot|x|^{\gamma}
\]
where $\mathbbm{s}$ is given by~\eqref{D:mathbbm-s}. Then we have  
\[
 H_{\overline{\mathbbm{s}}}(v,v_*) 
 =(2\pi)^{-3}\int_{\mathbb{R}^3}e^{iv\cdot\xi}\;\mathbb{A}_{\overline{\mathbbm{s}}}(v-v_*,\xi)\;\widehat{h}(\xi)\;d\xi
\]  
where 
\begin{equation}\label{A-v-v-*-w}
\begin{split}
\mathbb{A}_{\overline{\mathbbm{s}}}(v-v_*,\xi)&=\frac{|v-v_*|^{\gamma}_{\overline{\mathbbm{s}}}}{\lr{v}^{\gamma}\lr{v_*}^{\gamma}}
\int_{S^2_+} b(\cos\theta)\;e^{-i((v-v_*)\cdot\omega)(\xi\cdot\omega)} d\Omega(\omega) \\
&\eqdef (1-\mathbbm{s})(|v-v_*|)\;\mathbb{A}(v-v_*,\xi).
\end{split}
\end{equation}
From the definitions of $\mathbbm{s}$,\;\eqref{D:chi-k}~ and~\eqref{D:dyadic-partition}, we have the decomposition
\[
\begin{split}
&\mathbb{A}_{\overline{\mathbbm{s}}}(v-v_*,\xi)\\
&={\big (}\chi_{A}(v-v_*,\xi)+\chi_{B,1}(v-v_*,\xi)+\chi_{C,1}(v-v_*,\xi){\big )}\mathbb{A}(v-v_*,\xi)\\
&\eqdef \mathbb{A}_{A}(v-v_*,\xi)+\mathbb{A}_{B}(v-v_*,\xi)+\mathbb{A}_{C}(v-v_*,\xi)
\end{split}
\]
and the corresponding decomposition 
\[
H_{\overline{\mathbbm{s}}}(v,v_*)=H_{A}(v,v_*)+ H_{B}(v,v_*)+H_{C}(v,v_*).
\]
By the supports of $\chi_{A},\chi_{B,1}$ and $\chi_{C,1}$, we see that the estimates~\eqref{E:H-v-ineq} 
and~\eqref{E:H-v_*-ineq} follow from
\begin{equation}\label{E:H-v-v-*-region}
\begin{split}
&\sup\limits_{v_*}\|H_{A}(v,v_*) \|_{L^2(v)}\leq C \|h\|_{\dot{H}^{-\gamma}}\\
&\sup\limits_{v_*}\|H_{B}(v,v_*) \|_{L^2(v)}\leq C \|h\|_{L^2}\\
&\sup\limits_{v_*}\|H_{C}(v,v_*) \|_{L^2(v)}\leq C \|h\|_{L^2}
\end{split}
\end{equation}
and the other three  estimates which switches the roles of $v$ and $v_*$ of above estimates. By the argument of 
Lemma~\ref{L:T-v-v_*-ineq}, we know that it suffices to consider above three inequalities out of six. 
From the support of $\chi_A$, we know that the proof of  the first inequality of~\eqref{E:H-v-v-*-region}
may follow from the first case in the proof of Lemma~\ref{L:T-v-v_*-ineq}. 
Compare the definition of $\mathbb{A}_{\overline{\mathbbm{s}}}$ and $\mathbb{A}$ 
give by~\eqref{A-v-v-*} and note that 
the additional factor ${\lr{v}^{-\gamma}\lr{v_*}^{-\gamma}}$ does not affect the proof of first case 
there, i.e., the estimates for $p\pm(v-v_*,\xi)$ of~\eqref{E:expression} during the almost argument 
remain the same  after adding this factor, thus we conclude that result. 
For the second inequality of~\eqref{E:H-v-v-*-region}, we note 
that its proof may follow from Part II of the proof of Lemma~\ref{L:T-main}. The main idea there is 
that after using change of variables the operator can be written as FIOs on dyadic units. 
In the proof of Lemma~\ref{L:T-v-v_*-ineq} such argument works as well. Because $a(v-v_*,\xi)$  give 
by~\eqref{D:a-v-v-*-xi} contains two components which are both invariant under change of variables. More
precisely,  $|v-v_*|^{\gamma}|\xi|^{\gamma}
=|2^{-N}(v-v_*)|^{\gamma}|2^N\xi|^{\gamma}=|\ti{v}-\ti{v}_*|^{\gamma}|\ti{\xi}|^{\gamma}$ for any
$N\in\mathbb{N}$ and 
\begin{equation}\label{E:integration-invariant}
\int_{S^2_+} b(\cos\theta)\;e^{-i((v-v_*)\cdot\omega)(\xi\cdot\omega)} d\Omega(\omega)
=\int_{S^2_+} b(\cos\theta)\;e^{-i((\ti{v}-\ti{v}_*)\cdot\omega)(\ti{\xi}\cdot\omega)} d\Omega(\omega).
\end{equation}
When we applying change of variables argument to $\mathbb{A}_B(v-v_*,\xi)$, the integration part 
remains invariant as~\eqref{E:integration-invariant}. While the factor 
\[
\frac{|v-v_*|^{\gamma}_{\overline{\mathbbm{s}}}}{\lr{v}^{\gamma}\lr{v_*}^{\gamma}}
=\frac{|2^N(\ti{v}-\ti{v}_*)|^{\gamma}_{\overline{\mathbbm{s}}}}{\lr{2^N\ti{v}}^{\gamma}\lr{2^N\ti{v}_*}^{\gamma}}
\] 
combines the estimate of the later of~\eqref{E:integration-invariant} again induces the F.I.Os of order $0$.
Thus the arguments from Part II of the proof of Lemma~\ref{L:T-main} and Lemma~\ref{L:T-v-v_*-ineq} 
give us the second inequality of~\eqref{E:H-v-v-*-region}. To prove the third inequality 
of~\eqref{E:H-v-v-*-region}, we note that ${|v-v_*|^{\gamma}_{\overline{\mathbbm{s}}}}
/{\lr{v}^{-\gamma}\lr{v_*}^{-\gamma}}$ is bounded. Hence   
\[
|\mathbb{A}_{C}(v-v_*,\xi)|\leq C\chi_{C,1}(v-v_*,\xi).
\]
The above is parallel to~\eqref{E:kernel}, thus the argument of Part III of Lemma~\ref{L:T-main} 
gives us the desired inequality.
\end{proof}

\begin{proof}[Proof of Corollary~\ref{C:Q-L}]
For the homogeneous estimates, we note that in the proof of Lemma~\ref{L:T-main}, the operators
$T_{A}$ is a of order $\gamma-1$ instead of $0$. This means that if the reduction~\eqref{D:a-x-xi}
of $T$ is $T_A$ instead , we can replace $\gamma$ with $1$. Since the estimates $T_{B,1}$ follows 
from $T_A$, it enjoy the same property. And the operator $T_{C,1}$ is restricted to the low Fourier 
frequency, it is in our favor to rise the exponent. With these observations, we can derive the 
desired Lemma parallel to  Lemma~\ref{L:T-v-v_*-ineq} to conclude the result. 

The proof for inhomogeneous estimates is similarly, the key is that the first inequality of
~\eqref{E:H-v-v-*-region} can have $\dot{H}^{-1}$ in the left hand side due to the operator 
is defined by $\chi_A$. 
   
\end{proof}

\section{Some tools}

One of the most powerful tools in estimating the oscillatory integral
 $$I_{\Lambda}(u,f)=\int_{{\mathbb R}^n}e^{i\Lambda f(y)}u(y)dy,$$ for large $\Lambda$ is the 
following lemma of stationary phase asymptotics. There are several 
versions used widely, here we only record one of these which is 
from Theorem 7.7.5 of H\"{o}rmander~\cite{Hor83}. We use the notation $D_j=-i\partial_j$.
\begin{thm}[Stationary phase asymptotics]\label{L:stationary}
Let $K\subset{\mathbb R}^n$ be a compact set, $X$ an open neighborhood of $K$ and 
$k$ a positive number. 
If $u\in C^{2k}_{c}({\mathbb R}^n), f\in C^{3k+1}(X)$ and 
${\rm Im}\;f\geq 0$ in $X, {\rm Im}\;f(y_0)=0,
f'(y_0)=0,{\rm det}\;f''(y_0)\neq 0,f'\neq 0$ in $K\backslash\{{y_0}\}$ then
\begin{equation}\label{F:stationary-formula}
\begin{aligned}
|I&-e^{i\Lambda f(y_0)}({\rm det}(\Lambda f''(y_0)/2\pi i))^{-1/2}\sum_{j<k}\Lambda^{-j}L_ju|\\
  &\leq C\Lambda^{-k}\sum_{|\alpha|\leq 2k}\sup|D^{\alpha}u|\;,\;\Lambda>0
\end{aligned}
\end{equation}
Here $C$ is bounded when $f$ stays in a bounded set in $C^{3k+1}(X)$ and 
$|y-y_0|/|f'(y)|$ has a uniform bound. 
With $$g_{y_0}(y)=f(y)-f(y_0)-\langle f''(y_0)(y-y_0),y-y_0 \rangle/2$$ which 
vanish of third order at $x_0$ we have 
\begin{equation}\label{E:L-operator}
L_ju=\sum_{\nu-\mu=j}\;\;\sum_{2\nu\geq3\mu}i^{-j}2^{-\nu}\langle f''(y_0)^{-1}D,D\rangle^{\nu}
(g_{y_0}^{\mu}u)(y_0)/\mu!\nu!
\end{equation}
which is a differential operator of order $2j$ acting on $u$ at $y_0$. The coefficients are rational 
homogeneous functions of degree $-j$ in $f''(y_0),\cdots, f^{2j+2}(y_0)$ with denominator
$({\rm det}\; f''(y_0))^{3j}$. In every term the total number of derivatives of $u$ and $f''$ is at most $2j$. 
\end{thm}

In proof of lemma~\ref{lemma:Fourier_integral}, we need the following Schur test lemma(~\cite{HS78}, Theorem 5.2.). 
 See also Sogge's book~\cite{Sog93} Theorem 0.3.1
 for the related Young's inequality.
\begin{lem}[Schur test lemma]\label{lemma:Schur_test}
Let $X\;,\;Y$ be two measurable spaces. Let $T$ be an integral operator with the non-negative Schwartz kernel, i.e. 
$$Tf(x)=\int_{Y}K(x,y)f(y)dy.$$ 
If there exist functions $p(x)>0$ and $q(x)>0$ and numbers $\omega\;,\;\beta>0$ such that 
$$\int_Y K(x,y)q(y)dy\leq \omega p(x)$$ for almost all $x$ and
$$\int_X K(x,y)p(x)dx\leq \beta q(y)$$ for almost all $y$.
Then $T$ is a continuous operator $L^2 \rightarrow L^2$ with the operator norm
$$\|T\|_{L^2\rightarrow L^2}\leq \sqrt{\omega\beta}$$ 
\end{lem}

We also need the following Cotlar-Stein lemma.(See Stein's book~\cite{Ste93}) 
\begin{lem}[Cotlar-Stein lemma]\label{lemma:Coltar's_lemma}
Assume a family of $L^2$ bounded operators $\{T_j\}_{j\in Z^n}$ and a sequence of positive constants 
$\{\gamma(j)\}_{j\in Z^n}$ satisfy
$$\|T^{*}_iT_j\|_{L^2\rightarrow L^2}\leq \{\gamma(i-j)\}^2\;,\; \|T_iT^{*}_j\|_{L^2\rightarrow L^2}\leq \{\gamma(i-j)\}^2$$
and 
$$M=\sum_{j\in Z^n}\gamma(j)<\infty.$$ 
Then the operator $T=\sum_{j\in Z^n}T_j$ satisfies 
$$\|T\|_{L^2\rightarrow L^2}\leq M.$$
\end{lem}

\subsection*{Acknowledgment}
The research of author was supported in part by National Sci-Tech Grant MOST 106-2115-M-007-003, 
107-2923-M-001-001 and National Center for Theoretical Sciences.

\end{document}